\documentclass[12pt]{amsart}
\usepackage{amssymb}
\usepackage{amsmath}
\usepackage{amsthm}
\usepackage{graphicx}
\usepackage{enumerate}
\usepackage{tikz}
\usepackage{cleveref}
\usetikzlibrary{arrows}
\usetikzlibrary{arrows.meta}
\tikzset{>={Latex[width=1.2mm,length=1.7mm]}}
\usepackage{tabularx}
\usepackage{mathrsfs}
\newtheorem{thm}{Theorem}[section]
\newtheorem{prop}[thm]{Proposition}
\newtheorem{cor}[thm]{Corollary}

\newtheorem{obs}[thm]{Observation}
\newtheorem{conj}[thm]{Conjecture}
\newtheorem{prob}[thm]{Problem}
\newtheorem{lem}[thm]{Lemma}
\theoremstyle{definition}
\newtheorem{defn}[thm]{Definition}%[section]
\newtheorem{alg}[thm]{Algorithm}
\newtheorem{quest}[thm]{Question}

\numberwithin{equation}{section}

\newcommand{\sn}{\mathfrak{S}_n}

\newcommand{\mfs}[1]{\mathfrak{S}_{#1}}

\newcommand{\bn}{\mathfrak{B}_n}

\newcommand{\zsn}{\mathbb{Z}[\sn]}

\newcommand{\zqq}{\mathbb{Z}[\qp12, \qm12]}

\newcommand{\qp}[2]{q^{\frac{#1}{#2}}}
\newcommand{\qm}[2]{q^{\negthinspace\Bar\,\frac{#1}{#2}}}

\newcommand{\ol}[1]{\overline{#1}}

\newcommand{\hnq}{H_n(q)}

\newcommand{\A}{\mathcal{A}}
\newcommand{\B}{\mathcal{B}}
\newcommand{\C}{\mathcal{C}}
\newcommand{\D}{\mathcal{D}}

\newcommand{\pgap}{\mathrm{gap_{3412}}}
\newcommand{\singdeg}{\mathrm{singdeg}}

\newcommand{\icoh}{\mathrm{IH}}

\newcommand{\Wejm}{W^{\emptyset,J}_-}

\newcommand{\wtc}[2]{\widetilde{C}_{#1}}
\newcommand{\inv}{\textsc{inv}}

\newcommand{\defeq}{:=} 
\newcommand{\dfct}{\mathrm{dfct}}

\newcommand{\spn}{\mathrm{span}}

\newcommand{\src}[2]{\mathrm{src}_{{#1}, {#2}}}
\newcommand{\snk}[2]{\mathrm{snk}_{{#1}, {#2}}}
\newcommand{\ctr}[1]{{x}_{#1}}
\newcommand{\trunc}[1]{\mathrm{trunc}_{#1}}
\newcommand{\hgt}{\mathrm{height}}

\newcommand{\type}{\mathrm{type}}

\newcommand{\avoidsp}{avoids the patterns $3412$ and $4231${}}
\newcommand{\avoidp}{avoid the patterns $3412$ and $4231${}}
\newcommand{\avoidingp}{avoiding the patterns $3412$ and $4231${}}

\newcommand{\ntnsp}{\negthinspace}
\newcommand{\ntksp}{\negthickspace}
\newcommand{\nTksp}{\negthickspace\negthickspace}

\newcommand{\bp}{\begin{prob}}
\newcommand{\ep}{\end{prob}}

\newcommand{\ssm}{\smallsetminus}

\newcommand{\net}[1]{\mathcal G_{#1}}
\newcommand{\cnet}[1]{\mathcal F^\bullet_{#1}}

\newcommand{\ocnet}[1]{\mathcal F_{#1}}%{\mathcal F^{\circ\bullet}_{#1}}

\newcommand{\msfA}{\mathsf{A}}

\newcommand{\vertex}[1]{{#1}}
\newcommand{\phsum}{\phantom{\sum_A}}
\newcommand{\phint}{\phantom{\int}}
\newcommand{\phc}{\phantom{\wtc wq}}
\newcommand{\phmat}{\phantom{\begin{matrix}1\\1\\1\\1\end{matrix}}}
\newcommand{\phm}{\phantom M}

\newcommand{\schub}[1]{\Omega_{#1}}

\newcommand{\Fdbcae}{
    \begin{tikzpicture}[scale=.5,baseline=40]
  \pgfmathsetmacro{\indexOffset}{0.4}
  \foreach \index in {1,...,5} {
    \fill (0, \index) circle (1.3mm);
    \node at (0 - \indexOffset, \index) {$\scriptstyle \index$};
    \fill (3, \index) circle (1.3mm);
    \node at (3 + \indexOffset, \index) {$\scriptstyle \index$}; }
    \draw[-, thick] (0, 1) -- (1, 2) -- (2, 4) -- (3, 4);
    \draw[-, thick] (0, 2) -- (1, 1) -- (2, 1) -- (3, 2);
    \draw[-, thick] (0, 3) -- (1, 3) -- (3, 3) -- (3, 3);
    \draw[-, thick] (0, 4) -- (1, 4) -- (2, 2) -- (3, 1);
    \draw[-, thick] (0, 5) -- (1, 5) -- (2, 5) -- (3, 5);
    \fill (0.5, 1.5) circle (1.3mm);  
    \fill (1.5, 3) circle (1.3mm); 
    \fill (2.5, 1.5) circle (1.3mm);  
\end{tikzpicture}}

\newcommand{\Febcad}{
    \begin{tikzpicture}[scale=.5,baseline=40]
  \pgfmathsetmacro{\indexOffset}{0.4}
  \foreach \index in {1,...,5} {
    \fill (0, \index) circle (1.3mm);
    \node at (0 - \indexOffset, \index) {$\scriptstyle \index$};
    \fill (4, \index) circle (1.3mm);
    \node at (4 + \indexOffset, \index) {$\scriptstyle \index$}; }
    \draw[-, thick] (0, 1) -- (1, 2) -- (2, 4) -- (3, 4) -- (4, 5);
    \draw[-, thick] (0, 2) -- (1, 1) -- (2, 1) -- (3, 2) -- (4, 2);
    \draw[-, thick] (0, 3) -- (1, 3) -- (3, 3) -- (3, 3) -- (4, 3);
    \draw[-, thick] (0, 4) -- (1, 4) -- (2, 2) -- (3, 1) -- (4, 1);
    \draw[-, thick] (0, 5) -- (1, 5) -- (2, 5) -- (3, 5) -- (4, 4);
    \fill (0.5, 1.5) circle (1.3mm);  
    \fill (1.5, 3) circle (1.3mm); 
    \fill (2.5, 1.5) circle (1.3mm);  
    \fill (3.5, 4.5) circle (1.3mm);
\end{tikzpicture}}

\newcommand{\Fceadb}{
    \begin{tikzpicture}[scale=.5,baseline=40]
  \pgfmathsetmacro{\indexOffset}{0.4}
  \foreach \index in {1,...,5} {
    \fill (0, \index) circle (1.3mm);
    \node at (0 - \indexOffset, \index) {$\scriptstyle \index$};
    \fill (4, \index) circle (1.3mm);
    \node at (4 + \indexOffset, \index) {$\scriptstyle \index$}; }
    \draw[-, thick] (0, 1) -- (1, 1) -- (2, 2) -- (3, 2) -- (4, 3);
    \draw[-, thick] (0, 2) -- (1, 3) -- (2, 3) -- (3, 5) -- (4, 5);
    \draw[-, thick] (0, 3) -- (1, 2) -- (2, 1) -- (3, 1) -- (4, 1);
    \draw[-, thick] (0, 4) -- (1, 4) -- (2, 4) -- (3, 4) -- (4, 4);
    \draw[-, thick] (0, 5) -- (1, 5) -- (2, 5) -- (3, 3) -- (4, 2);
    \fill (0.5, 2.5) circle (1.3mm);  
    \fill (1.5, 1.5) circle (1.3mm); 
    \fill (2.5, 4) circle (1.3mm);  
    \fill (3.5, 2.5) circle (1.3mm);
\end{tikzpicture}}

\newcommand{\Fcebda}{
    \begin{tikzpicture}[scale=.5,baseline=40]
  \pgfmathsetmacro{\indexOffset}{0.4}
  \foreach \index in {1,...,5} {
    \fill (0, \index) circle (1.3mm);
    \node at (0 - \indexOffset, \index) {$\scriptstyle \index$};
    \fill (3, \index) circle (1.3mm);
    \node at (3 + \indexOffset, \index) {$\scriptstyle \index$}; }
    \draw[-, thick] (0, 1) -- (1, 1) -- (2, 1) -- (3, 3);
    \draw[-, thick] (0, 2) -- (1, 3) -- (2, 5) -- (3, 5);
    \draw[-, thick] (0, 3) -- (1, 2) -- (3, 2);
    \draw[-, thick] (0, 4) -- (3, 4);
    \draw[-, thick] (0, 5) -- (1, 5) -- (2, 3) -- (3, 1);
    \fill (0.5, 2.5) circle (1.3mm);  
    \fill (1.5, 4) circle (1.3mm); 
    \fill (2.5, 2) circle (1.3mm);  
\end{tikzpicture}}

\newcommand{\Fcedab}{
    \begin{tikzpicture}[scale=.5,baseline=40]
  \pgfmathsetmacro{\indexOffset}{0.4}
  \foreach \index in {1,...,5} {
    \fill (0, \index) circle (1.3mm);
    \node at (0 - \indexOffset, \index) {$\scriptstyle \index$};
    \fill (4, \index) circle (1.3mm);
    \node at (4 + \indexOffset, \index) {$\scriptstyle \index$}; }
    \draw[-, thick] (0, 1) -- (1, 1) -- (2, 1) -- (3, 2) -- (4, 3);
    \draw[-, thick] (0, 2) -- (0.5, 3) -- (1.5, 4) -- (2, 5) -- (4, 5);
    \draw[-, thick] (0, 3) -- (0.5, 3) -- (1.5, 4) -- (2, 4) -- (4, 4);
    \draw[-, thick] (0, 4) -- (1, 2) -- (2, 2) -- (3, 1) -- (4, 1);
    \draw[-, thick] (0, 5) -- (1, 5) -- (2, 3) -- (3, 3) -- (4, 2);
\node at (.75, 3.85) {$\scriptstyle _{(2)}$};
    \fill (0.5, 3) circle (1.3mm);  
    \fill (1.5, 4) circle (1.3mm); 
    \fill (2.5, 1.5) circle (1.3mm);  
    \fill (3.5, 2.5) circle (1.3mm);
\end{tikzpicture}}

\newcommand{\Febdca}{
    \begin{tikzpicture}[scale=.5,baseline=40]
  \pgfmathsetmacro{\indexOffset}{0.4}
  \foreach \index in {1,...,5} {
    \fill (0, \index) circle (1.3mm);
    \node at (0 - \indexOffset, \index) {$\scriptstyle \index$};
    \fill (3, \index) circle (1.3mm);
    \node at (3 + \indexOffset, \index) {$\scriptstyle \index$}; }
    \draw[-, thick] (0, 1) -- (1, 2) -- (2, 5) -- (3, 5);
    \draw[-, thick] (0, 2) -- (1, 1) -- (2, 1) -- (3, 2);
    \draw[-, thick] (0, 3) -- (1, 3) -- (2, 4) -- (3, 4);
    \draw[-, thick] (0, 4) -- (1, 4) -- (2, 3) -- (3, 3);
    \draw[-, thick] (0, 5) -- (1, 5) -- (2, 2) -- (3, 1);
    \fill (0.5, 1.5) circle (1.3mm);  
    \fill (1.5, 3.5) circle (1.3mm); 
    \fill (2.5, 1.5) circle (1.3mm);  
\end{tikzpicture}}

\newcommand{\Fecdba}{
    \begin{tikzpicture}[scale=.5,baseline=40]
  \pgfmathsetmacro{\indexOffset}{0.4}
  \foreach \index in {1,...,5} {
    \fill (0, \index) circle (1.3mm);
    \node at (0 - \indexOffset, \index) {$\scriptstyle \index$};
    \fill (3, \index) circle (1.3mm);
    \node at (3 + \indexOffset, \index) {$\scriptstyle \index$}; }
    \draw[-, thick] (0, 1) -- (1, 1) -- (2, 4) -- (3, 5);
    \draw[-, thick] (0, 2) -- (1, 2) -- (2, 3) -- (3, 3);
    \draw[-, thick] (0, 3) -- (1, 5) -- (2, 5) -- (3, 4);
    \draw[-, thick] (0, 4) -- (0.5, 4) -- (1.5, 2.5) -- (2, 2) -- (3, 2);
    \draw[-, thick] (0, 5) -- (0.5, 4) -- (1.5, 2.5) -- (2, 1) -- (3, 1);
\node at (1.25, 3.7) {$\scriptstyle _{(2)}$};
    \fill (0.5, 4) circle (1.3mm);  
    \fill (1.5, 2.5) circle (1.3mm); 
    \fill (2.5, 4.5) circle (1.3mm);  
\end{tikzpicture}}

\newcommand{\Febcda}{
    \begin{tikzpicture}[scale=.5,baseline=40]
  \pgfmathsetmacro{\indexOffset}{0.4}
  \foreach \index in {1,...,5} {
    \fill (0, \index) circle (1.3mm);
    \node at (0 - \indexOffset, \index) {$\scriptstyle \index$};
    \fill (5, \index) circle (1.3mm);
    \node at (5 + \indexOffset, \index) {$\scriptstyle \index$}; }
    \draw[-, thick] (0, 1) -- (1, 2) -- (2, 2) -- (3, 4) -- (4, 4) -- (5, 5);
    \draw[-, thick] (0, 2) -- (1, 1) -- (2, 1) -- (3, 1) -- (4, 2) -- (5, 2);
    \draw[-, thick] (0, 3) -- (1, 3) -- (2, 3) -- (3, 3) -- (4, 3) -- (5, 3);
    \draw[-, thick] (0, 4) -- (1, 4) -- (2, 5) -- (3, 5) -- (4, 5) -- (5, 4);
    \draw[-, thick] (0, 5) -- (1, 5) -- (2, 4) -- (3, 2) -- (4, 1) -- (5, 1);
    \fill (0.5, 1.5) circle (1.3mm);  
    \fill (1.5, 4.5) circle (1.3mm); 
    \fill (2.5, 3) circle (1.3mm);  
    \fill (3.5, 1.5) circle (1.3mm);
    \fill (4.5, 4.5) circle (1.3mm);
\end{tikzpicture}}

\newcommand{\Fdebca}{
    \begin{tikzpicture}[scale=.5,baseline=40]
  \pgfmathsetmacro{\indexOffset}{0.4}
  \foreach \index in {1,...,5} {
    \fill (0, \index) circle (1.3mm);
    \node at (0 - \indexOffset, \index) {$\scriptstyle \index$};
    \fill (4, \index) circle (1.3mm);
    \node at (4 + \indexOffset, \index) {$\scriptstyle \index$}; }
    \draw[-, thick] (0, 1) -- (1, 1) -- (2, 1) -- (3, 3) -- (4, 4);
    \draw[-, thick] (0, 2) -- (1, 3) -- (2, 5) -- (3, 5) -- (4, 5);
    \draw[-, thick] (0, 3) -- (1, 2) -- (2, 2) -- (3, 2) -- (4, 2);
    \draw[-, thick] (0, 4) -- (1, 4) -- (2, 4) -- (3, 4) -- (4, 3);
    \draw[-, thick] (0, 5) -- (1, 5) -- (2, 3) -- (3, 1) -- (4, 1);
    \fill (0.5, 2.5) circle (1.3mm);  
    \fill (1.5, 4) circle (1.3mm); 
    \fill (2.5, 2) circle (1.3mm);  
    \fill (3.5, 3.5) circle (1.3mm);
\end{tikzpicture}}

\def\hhhsp{\def\baselinestretch{0.125}\large\normalsize}

\def\ssp{\def\baselinestretch{1.0}\large\normalsize}
\begin{document}
\author{Tommy Parisi}
\author{Mark Skandera}
\author{Ben Spahiu}
\author{Jiayuan Wang}
%\author{Justin Lambright}%\corref{cor}}   needed for elsarticle
%\ead{jjlambright@anderson.edu}
%\author{Jongwon Kim and Mark Skandera}
%\ead{mas906@lehigh.edu,mark.skandera@gmail.com}
\title[Kahzdan--Lusztig basis elements having no reversal factorization]
%[Reversal factorization of Kazhdan--Lusztig basis elements]
{On Kazhdan--Lusztig basis elements having no reversal factorization}%{On the impossibility of reversal factorization of some Kazhdan--Lusztig basis elements}
%The natural expansion of products Combinatorial formulas for Hecke algebra multiplication}
%at Kazhdan--Lusztig
% basis elements}
%\title[Induced sign characters]

\bibliographystyle{dart}

\date{\today}

\begin{abstract}
%%% Abstract from Defects 2 file:
% Let $H$ be the Iwahori--Hecke algebra corresponding to any Coxeter group.
%   
%An extension [$\star$ Clearwater-S]

For $w$ in the symmetric group $\sn$, let $\wtc wq$
%= \sum P_{v,w}(q) T_v$ 
be the corresponding modified, signless Kazhdan--Lusztig basis element of the type-$\msfA$ Hecke algebra $\hnq$. %indexed by $w$.
%the permutation 
%$w \in \sn$.
%Clearwater and the second author's 
An extension 
[{\em Ann.\;Comb.}\;{\bf 25}, no.\,3 (2021) pp.~757--787] of a result 
of Deodhar
%{\em defect} statistic 
[{\em Geom.\;Dedicata} {\bf 36}, (1990) pp.~95--119] implies that 
%if 
%a $\mathbb Z[q]$-multiple of 
%$\wtc wq$ factors as 
any factorization of the form
\begin{equation*}
    \wtc wq = \frac1{f(q)} \wtc{v^{(1)}}q \cdots \wtc{v^{(r)}}q,
    \end{equation*}
with $v^{(1)},\dotsc,v^{(r)}$ maximal elements of parabolic subgroups of 
%the symmetric group 
$\sn$ and $f(q) \in \mathbb N[q]$ depending on these,
%corresponding
%to smooth Schubert varieties,
%then a certain planar nework
provides cancellation-free combinatorial
interpretations of the polynomials $\{P_{v,w}(q) \,|\, v \in \sn \}$ appearing
in the expansion $\sum_v P_{v,w}(q) T_v$ 
of $\wtc wq$ in terms of the natural basis $\{ T_v \,|\, v \in \sn \}$ of $\hnq$.
%the Hecke algebra.
While the set of permutations $w \in \sn$ admitting such a factorization of $\smash{\wtc wq}$ has not yet been characterized,
we apply a result of Gaetz -- Gao 
[{\em Adv. Math.}
{\bf 457} (2024)
Paper No. 109941] 
%of Gaetz -- Gao 
to describe a set admitting 
%for which 
no such factorization.
%exists.  %($\star$ Replace $\wtc wq$ with $\widetilde{C}_w$?)
\end{abstract}
\maketitle

\section{Introduction}\label{s:introduction}
The {\em Kazhdan--Lusztig polynomials} $\{ P_{v,w}(q) \,|\, v, w \in \sn \} \subset \mathbb N[q]$ are entries of the change-of-basis matrix relating a certain  {\em Kazhdan--Lusztig basis} of the Hecke algebra with another {\em natural basis.}
First appearing in the study of representations of the Hecke algebra, 
they were given existential and recursive definitions in \cite{KLRepCH}.
Appearances of the polynomials in other areas such as Lie Theory~\cite{BeilBern}, \cite{BeilBernPfofJantzen}, \cite{BryKash},
%\cite{BeilBern}, \cite{BeilBernPfofJantzen}, \cite{BryKash},
quantum groups~\cite{FKK}, 
combinatorics~\cite{HaimanHecke},
and Schubert varieties~\cite{KLRepCH}, \cite{KLSchub} 
have inspired a search for simpler descriptions.
%than those in \cite{KLRepCH}.
Ideally, such a description should interpret each coefficient of $P_{v,w}(q)$ as a set cardinality.

Some 
famous 
alternative formulas for the Kazhdan--Lusztig polynomials are due to Brenti and Deodhar. Brenti expressed $P_{v,w}(q)$ in two different ways
%~\cite[Thm.\,4.1]{Brenti94}, \cite[Thm.\,4.5]{Brenti97} 
~\cite[\S 3]{Brenti94}, \cite[\S 3]{Brenti97} 
as simple linear combinations of recursively defined polynomials in $\mathbb Z[q]$ having both positive and negative 
coefficients.
%~\cite[\S 3]{Brenti94}, \cite[\S 3]{Brenti97}.
Because of negative coefficients and recursive definitions, these formulas do not interpret coefficients in $P_{v,w}(q)$ as set cardinalities. Deodhar~\cite{Deodhar90} developed an algorithm which takes any reduced expression for $w$ as an input, and 
%produces
outputs 
a set $\mathscr E_{\min}$ of (not necessarily reduced) expressions
for other permutations in $\sn$. For each $v \in \sn$ and $k > 0$, the coefficient of $q^k$ in $P_{v,w}(q)$ is equal to the cardinality of a certain subset of $\mathscr E_{\min}$. On the other hand, the algorithmic component of Deodhar's method makes it difficult to apply his combinatorial interpretation in practice.
%(For more progress, see \cite{DyerGKLP}, \cite{LS11}, \cite{ZelSmallResEng}.)

%($\star$ Search for and include more recent results as well.)

Billey and Warrington showed~\cite[Thm.\,1, Rmk.\,6]{BWHex} that 
when $w$ has certain properties, Deodhar's algorithm is trivial, %terminates immediately 
and the output set 
$\mathscr E_{\min}$ of expressions can be replaced by a more visually appealing set of path families in a certain wiring diagram.  
Again for each $v$ and $k$, the coefficient of $q^k$ in $P_{v,w}(q)$ is equal to the cardinality of a subset of these path families. Clearwater and the second author~\cite[Cor.\,5.3]{CSkanTNNChar} then extended this result to permutations $w$ for which the Kazhdan--Lusztig basis element $\smash{\wtc wq}$ factors nicely, but did not solve the problem~\cite[Quest.\,4.5]{SkanNNDCB} of characterizing such permutations $w$. 

In Sections~\ref{s:snplanar} -- \ref{s:hnqplanar} we review basic facts and results about
%relate
the symmetric group, planar networks, the Hecke algebra, and the Kazhdan--Lusztig basis and polynomials.
%review basic facts about the symmetric group, planar networks, Hecke algebra, and 
%we define 
%the Kazhdan--Lusztig basis and polynomials. 
In Section~\ref{s:defectremoval} we use the result~\cite[Cor.\,5.3]{CSkanTNNChar} to state properties of polynomials which arise in the natural expansion of products of certain Kazhdan--Lusztig basis elements of the Hecke algebra. This leads to a partial answer in Section~\ref{s:nofactor} to the characterization question~\cite[Quest.\,4.5]{SkanNNDCB}: a description of 
%by describing 
certain Kazhdan--Lusztig basis elements which do not factor as desired.

\section{The symmetric group and planar networks}\label{s:snplanar}

Let $\sn$ be the symmetric group, with standard generators
$s_1,\dotsc,s_{n-1}$, length function $\ell$, and Bruhat order $\leq$.
(See, e.g., \cite{BBCoxeter} for definitions.)
Given a word $u = u_1 \cdots u_k$ in $\mfs k$,
and a word $y = y_1 \cdots y_k$ having $k$ distinct letters,
we say that {\em $y$ matches the pattern $u$} if the letters of $y$ appear
in the same relative order as those of $u$; that is, if we have
$u_i < u_j$ if and only if $y_i < y_j$ for all $i,j \in [k] \defeq \{1,\dotsc,k\}$.
On the other hand, 
say that $w \in \mfs n$ 
{\em avoids the pattern $u$} if no subword
%of the one-line notation
of $w$ matches the pattern $u$.

%For example, it 
It is easy to see that for 
each subinterval $[a,b] \defeq \{a,\dotsc,b\}$ of $[n]$,
%$a < b$, 
the 
%permutation $w$ satsifying 
{\em reversal} 
\begin{equation}\label{eq:reversaldef} s_{[a,b]} \defeq 1 \cdots (a-1) b \cdots a (b+1) \cdots n \in \sn
\end{equation}
\avoidsp.
This element
%The reversal $s_{[a,b]}$ 
is the unique longest (maximum length) element of the subgroup of $\sn$ generated by $s_a, \dotsc, s_{b-1}$.
More generally, each {\em parabolic} subgroup of $\sn$ generated by a subset 
%$I$ 
of generators 
has longest element equal to a product of reversals on disjoint intervals.
Multiplication of reversals in $\sn$ or of related elements
\begin{equation}\label{eq:Dab}
D_{[a,b]} \defeq \ntksp 
\sum_{v \leq s_{[a,b]}}\ntksp  v
\end{equation}
in $\zsn$ can be performed graphically with certain planar networks.

Define a {\em planar network of order n} to be a directed, planar, acyclic multigraph with $2n$ boundary vertices having $n$ source vertices on the left and $n$ sink vertices on the right, both labeled $1, \dotsc, n$ from bottom to top.
We will allow edges $(\vertex x, \vertex y)$
to be marked by a positive integer multiplicity $\mu(x,y)$.
Let $\net n$ denote the set of such networks.
For each subinterval $[a,b]$ of $[n]$ we define a {\em simple star network}
%$G_{[a,b]} = \smash{G_{[a,b]}^{[h,l]}} \in \net{A}{[h,l]}$ by
$\smash{F_{[a,b]}} \in \net n$ by
%In particular, we associate to each reversal $s_{[a,b]} \in \sn$ 
%a {\em simple star network} $\smash{F_{[a,b]}}$ bt
\begin{enumerate}
    \item an interior vertex $x$ lies between the sources and sinks,
    \item for $i \in [a,b]$ we have edges from source $i$ to $x$ and from $x$ to sink $i$,
    \item for $i \notin [a,b]$ we have edges from source $i$ to sink $i$.
\end{enumerate}
For example, the simple star network $F_{\smash{[2,4]}} \in \net 4$ is
%above can be more completely drawn as
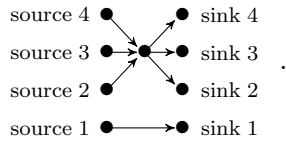
\begin{equation*}
\begin{tikzpicture}[scale=.5,baseline=15]
\node at (-2.5,2.5) {$\scriptstyle{ \mathrm{source}\;4}$};
\node at (-2.5,1.5) {$\scriptstyle{ \mathrm{source}\;3}$};
\node at (-2.5,0.5) {$\scriptstyle{ \mathrm{source}\;2}$};
\node at (-2.5,-0.5) {$\scriptstyle{ \mathrm{source}\;1}$};  
\node at (2.25,2.5) {$\scriptstyle{ \mathrm{sink}\;4}$};
\node at (2.25,1.5) {$\scriptstyle{ \mathrm{sink}\;3}$};
\node at (2.25,0.5) {$\scriptstyle{ \mathrm{sink}\;2}$};
\node at (2.25,-0.5) {$\scriptstyle{ \mathrm{sink}\;1}$};  
\draw[->,>=stealth'] (-1,2.5) -- (-.2,1.65);
\draw[->,>=stealth'] (-1,1.5) -- (-.2,1.5);
\draw[->,>=stealth'] (-1,0.5) -- (-.2,1.35);
%\draw[->,>=stealth'] (-1,-0.5) -- (-.2,-1.35);
%\draw[->,>=stealth'] (-1,-1.5) -- (-.2,-1.5);
%\draw[->,>=stealth'] (-1,-2.5) -- (-.2,-1.65);
\draw[->,>=stealth'] (0,1.5) -- (.8,2.35);
\draw[->,>=stealth'] (0,1.5) -- (.8,1.5);
\draw[->,>=stealth'] (0,1.5) -- (.8,0.65);
%\draw[->,>=stealth'] (0,-1.5) -- (.8,-2.35);
%\draw[->,>=stealth'] (0,-1.5) -- (.8,-1.5);
%\draw[->,>=stealth'] (0,-1.5) -- (.8,-0.65);
\draw[->,>=stealth'] (-1,-0.5) -- (.8,-0.5);
\node at (-1,2.5) {$\bullet$}; 
\node at (-1,1.5) {$\bullet$}; 
\node at (-1,0.5) {$\bullet$}; 
\node at (-1,-0.5) {$\bullet$};
%\node at (-1,-1.5) {$\bullet$};
%\node at (-1,-2.5) {$\bullet$};
\node at (1,2.5) {$\bullet$}; 
\node at (1,1.5) {$\bullet$}; 
\node at (1,0.5) {$\bullet$}; 
\node at (1,-0.5) {$\bullet$};
%\node at (1,-1.5) {$\bullet$};
%\node at (1,-2.5) {$\bullet$};
\node at (0,1.5) {$\bullet$};
%\node at (0,-1.5) {$\bullet$};
\end{tikzpicture}
\ .
\end{equation*}  
For economy, we will omit edge orientations
%, vertices, 
and the words "source" and "sink" from figures.   
Thus
%drawings of planar networks.  For example,
the seven simple star networks in $\net 4$ are
\begin{equation}\label{eq:simplestarnets}
%  \begin{tikzpicture}[scale=.5,baseline=0]
%  \node at (0,1.5) {$(2)$};     
%  \node at (0,0.5) {$(1)$};     
%  \node at (0,-0.5) {$(\ol 1)$};
%  \node at (0,-1.5) {$(\ol 2)$};
%  \end{tikzpicture}
% \qquad \qquad
\begin{gathered}
   \begin{tikzpicture}[scale=.6,baseline=-25]
\node at (-.4,0) {$\scriptstyle 4$};
\node at (-.4,-1) {$\scriptstyle 3$};
\node at (-.4,-2) {$\scriptstyle{2}$};
\node at (-.4,-3) {$\scriptstyle{1}$};  
\node at (1.4,0) {$\scriptstyle 4$};
\node at (1.4,-1) {$\scriptstyle 3$};
\node at (1.4,-2) {$\scriptstyle{2}$};
\node at (1.4,-3) {$\scriptstyle{1}$};  
\draw[-] (0,0) -- (1,-3);
\draw[-] (0,-1) -- (1,-2);
\draw[-] (0,-2) -- (1,-1);
\draw[-] (0,-3) -- (1,0);
\fill (0,0) circle  (1mm); \fill (0,-1) circle  (1mm); \fill (0,-2) circle  (1mm); \fill (0,-3) circle  (1mm);
\fill (1,0) circle  (1mm); \fill (1,-1) circle  (1mm); \fill (1,-2) circle  (1mm); \fill (1,-3) circle  (1mm);
\fill (.5, -1.5) circle (1mm);
\end{tikzpicture}%,
,\\
\phantom{\sum}\ntksp F_{[1,4]}\phantom{\sum}
\end{gathered}
\phm
\begin{gathered}
\begin{tikzpicture}[scale=.6,baseline=-25]
\node at (-.4,0) {$\scriptstyle 4$};
\node at (-.4,-1) {$\scriptstyle 3$};
\node at (-.4,-2) {$\scriptstyle{2}$};
\node at (-.4,-3) {$\scriptstyle{1}$};  
\node at (1.4,0) {$\scriptstyle 4$};
\node at (1.4,-1) {$\scriptstyle 3$};
\node at (1.4,-2) {$\scriptstyle{2}$};
\node at (1.4,-3) {$\scriptstyle{1}$}; 
%\node at (-.4,0) {$\scriptstyle 2$};
%\node at (-.4,-1) {$\scriptstyle 1$};
%\node at (-.4,-2) {$\scriptstyle{\ol1}$};
%\node at (-.4,-3) {$\scriptstyle{\ol2}$};  
%\node at (1.4,0) {$\scriptstyle 2$};
%\node at (1.4,-1) {$\scriptstyle 1$};
%\node at (1.4,-2) {$\scriptstyle{\ol1}$};
%\node at (1.4,-3) {$\scriptstyle{\ol2}$};  
\draw[-] (0,0) -- (1,-2);
\draw[-] (0,-1) -- (1,-1);
\draw[-] (0,-2) -- (1,0);
\draw[-] (0,-3) -- (1,-3);
\fill (0,0) circle  (1mm); \fill (0,-1) circle  (1mm); \fill (0,-2) circle  (1mm); \fill (0,-3) circle  (1mm);
\fill (1,0) circle  (1mm); \fill (1,-1) circle  (1mm); \fill (1,-2) circle  (1mm); \fill (1,-3) circle  (1mm);
\fill (.5, -1) circle (1mm);
\end{tikzpicture},\\
   \phantom{\sum}\ntksp F_{[2,4]}\phantom{\sum}
 \end{gathered}
\phm
\begin{gathered}
\begin{tikzpicture}[scale=.6,baseline=-25]
\node at (-.4,0) {$\scriptstyle 4$};
\node at (-.4,-1) {$\scriptstyle 3$};
\node at (-.4,-2) {$\scriptstyle{2}$};
\node at (-.4,-3) {$\scriptstyle{1}$};  
\node at (1.4,0) {$\scriptstyle 4$};
\node at (1.4,-1) {$\scriptstyle 3$};
\node at (1.4,-2) {$\scriptstyle{2}$};
\node at (1.4,-3) {$\scriptstyle{1}$}; 
%\node at (-.4,0) {$\scriptstyle 2$};
%\node at (-.4,-1) {$\scriptstyle 1$};
%\node at (-.4,-2) {$\scriptstyle{\ol1}$};
%\node at (-.4,-3) {$\scriptstyle{\ol2}$};  
%\node at (1.4,0) {$\scriptstyle 2$};
%\node at (1.4,-1) {$\scriptstyle 1$};
%\node at (1.4,-2) {$\scriptstyle{\ol1}$};
%\node at (1.4,-3) {$\scriptstyle{\ol2}$};  
\draw[-] (0,0) -- (1,0);
\draw[-] (0,-1) -- (1,-3);
\draw[-] (0,-2) -- (1,-2);
\draw[-] (0,-3) -- (1,-1);
\fill (0,0) circle  (1mm); \fill (0,-1) circle  (1mm); \fill (0,-2) circle  (1mm); \fill (0,-3) circle  (1mm);
\fill (1,0) circle  (1mm); \fill (1,-1) circle  (1mm); \fill (1,-2) circle  (1mm); \fill (1,-3) circle  (1mm);
\fill (.5, -2) circle (1mm);
\end{tikzpicture},\\
   \phantom{\sum}\ntksp F_{[1,3]}\phantom{\sum}
 \end{gathered}
\phm
\begin{gathered}
\begin{tikzpicture}[scale=.6,baseline=-25]
\node at (-.4,0) {$\scriptstyle 4$};
\node at (-.4,-1) {$\scriptstyle 3$};
\node at (-.4,-2) {$\scriptstyle{2}$};
\node at (-.4,-3) {$\scriptstyle{1}$};  
\node at (1.4,0) {$\scriptstyle 4$};
\node at (1.4,-1) {$\scriptstyle 3$};
\node at (1.4,-2) {$\scriptstyle{2}$};
\node at (1.4,-3) {$\scriptstyle{1}$}; 
%\node at (-.4,0) {$\scriptstyle 2$};
%\node at (-.4,-1) {$\scriptstyle 1$};
%\node at (-.4,-2) {$\scriptstyle{\ol1}$};
%\node at (-.4,-3) {$\scriptstyle{\ol2}$};  
%\node at (1.4,0) {$\scriptstyle 2$};
%\node at (1.4,-1) {$\scriptstyle 1$};
%\node at (1.4,-2) {$\scriptstyle{\ol1}$};
%\node at (1.4,-3) {$\scriptstyle{\ol2}$};  
\draw[-] (0,0) -- (1,-1);
\draw[-] (0,-1) -- (1,0);
\draw[-] (0,-2) -- (1,-2);
\draw[-] (0,-3) -- (1,-3);
\fill (0,0) circle  (1mm); \fill (0,-1) circle  (1mm); \fill (0,-2) circle  (1mm); \fill (0,-3) circle  (1mm);
\fill (1,0) circle  (1mm); \fill (1,-1) circle  (1mm); \fill (1,-2) circle  (1mm); \fill (1,-3) circle  (1mm);
\fill (.5, -0.5) circle (1mm);
\end{tikzpicture},\\
   \phantom{\sum}\ntksp F_{[3,4]}\phantom{\sum}
 \end{gathered}
\phm
\begin{gathered}
\begin{tikzpicture}[scale=.6,baseline=-25]
\node at (-.4,0) {$\scriptstyle 4$};
\node at (-.4,-1) {$\scriptstyle 3$};
\node at (-.4,-2) {$\scriptstyle{2}$};
\node at (-.4,-3) {$\scriptstyle{1}$};  
\node at (1.4,0) {$\scriptstyle 4$};
\node at (1.4,-1) {$\scriptstyle 3$};
\node at (1.4,-2) {$\scriptstyle{2}$};
\node at (1.4,-3) {$\scriptstyle{1}$}; 
%\node at (-.4,0) {$\scriptstyle 2$};
%\node at (-.4,-1) {$\scriptstyle 1$};
%\node at (-.4,-2) {$\scriptstyle{\ol1}$};
%\node at (-.4,-3) {$\scriptstyle{\ol2}$};  
%\node at (1.4,0) {$\scriptstyle 2$};
%\node at (1.4,-1) {$\scriptstyle 1$};
%\node at (1.4,-2) {$\scriptstyle{\ol1}$};
%\node at (1.4,-3) {$\scriptstyle{\ol2}$};  
\draw[-] (0,0) -- (1,0);
\draw[-] (0,-1) -- (1,-2);
\draw[-] (0,-2) -- (1,-1);
\draw[-] (0,-3) -- (1,-3);
\fill (0,0) circle  (1mm); \fill (0,-1) circle  (1mm); \fill (0,-2) circle  (1mm); \fill (0,-3) circle  (1mm);
\fill (1,0) circle  (1mm); \fill (1,-1) circle  (1mm); \fill (1,-2) circle  (1mm); \fill (1,-3) circle  (1mm);
\fill (.5, -1.5) circle (1mm);
\end{tikzpicture},\\
   \phantom{\sum}\ntksp F_{[2,3]}\phantom{\sum}
 \end{gathered}
\phm
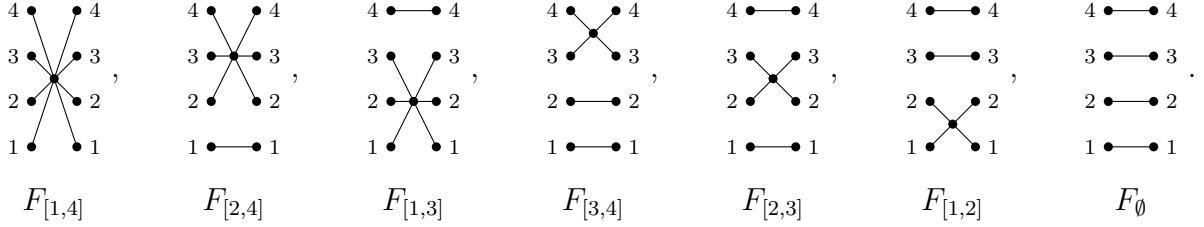
\begin{gathered}
\begin{tikzpicture}[scale=.6,baseline=-25]
\node at (-.4,0) {$\scriptstyle 4$};
\node at (-.4,-1) {$\scriptstyle 3$};
\node at (-.4,-2) {$\scriptstyle{2}$};
\node at (-.4,-3) {$\scriptstyle{1}$};  
\node at (1.4,0) {$\scriptstyle 4$};
\node at (1.4,-1) {$\scriptstyle 3$};
\node at (1.4,-2) {$\scriptstyle{2}$};
\node at (1.4,-3) {$\scriptstyle{1}$}; 
%\node at (-.4,0) {$\scriptstyle 2$};
%\node at (-.4,-1) {$\scriptstyle 1$};
%\node at (-.4,-2) {$\scriptstyle{\ol1}$};
%\node at (-.4,-3) {$\scriptstyle{\ol2}$};  
%\node at (1.4,0) {$\scriptstyle 2$};
%\node at (1.4,-1) {$\scriptstyle 1$};
%\node at (1.4,-2) {$\scriptstyle{\ol1}$};
%\node at (1.4,-3) {$\scriptstyle{\ol2}$};  
\draw[-] (0,0) -- (1,0);
\draw[-] (0,-1) -- (1,-1);
\draw[-] (0,-2) -- (1,-3);
\draw[-] (0,-3) -- (1,-2);
\fill (0,0) circle  (1mm); \fill (0,-1) circle  (1mm); \fill (0,-2) circle  (1mm); \fill (0,-3) circle  (1mm);
\fill (1,0) circle  (1mm); \fill (1,-1) circle  (1mm); \fill (1,-2) circle  (1mm); \fill (1,-3) circle  (1mm);
\fill (.5, -2.5) circle (1mm);
\end{tikzpicture},\\
   \phantom{\sum}\ntksp F_{[1,2]}\phantom{\sum}
 \end{gathered}
\
\phm
\begin{gathered}
\begin{tikzpicture}[scale=.6,baseline=-25]
\node at (-.4,0) {$\scriptstyle 4$};
\node at (-.4,-1) {$\scriptstyle 3$};
\node at (-.4,-2) {$\scriptstyle{2}$};
\node at (-.4,-3) {$\scriptstyle{1}$};  
\node at (1.4,0) {$\scriptstyle 4$};
\node at (1.4,-1) {$\scriptstyle 3$};
\node at (1.4,-2) {$\scriptstyle{2}$};
\node at (1.4,-3) {$\scriptstyle{1}$}; 
%\node at (-.4,0) {$\scriptstyle 2$};
%\node at (-.4,-1) {$\scriptstyle 1$};
%\node at (-.4,-2) {$\scriptstyle{\ol1}$};
%\node at (-.4,-3) {$\scriptstyle{\ol2}$};  
%\node at (1.4,0) {$\scriptstyle 2$};
%\node at (1.4,-1) {$\scriptstyle 1$};
%\node at (1.4,-2) {$\scriptstyle{\ol1}$};
%\node at (1.4,-3) {$\scriptstyle{\ol2}$};  
\draw[-] (0,0) -- (1,0);
\draw[-] (0,-1) -- (1,-1);
\draw[-] (0,-2) -- (1,-2);
\draw[-] (0,-3) -- (1,-3);
\fill (0,0) circle  (1mm); \fill (0,-1) circle  (1mm); \fill (0,-2) circle  (1mm); \fill (0,-3) circle  (1mm);
\fill (1,0) circle  (1mm); \fill (1,-1) circle  (1mm); \fill (1,-2) circle  (1mm); \fill (1,-3) circle  (1mm);
%\fill (.5, -1.5) circle (1mm);
\end{tikzpicture}.\\
\phantom{\sum}\ntksp F_\emptyset \phantom{\sum}
%\phantom{\sum}\ntksp G_{[\ol2,\ol2]}\phantom{\sum}
 \end{gathered}
\end{equation}

Given networks $E, F \in \net{n}$
%of order $n = |[h,l]|$,
in which all sources have
outdegree $1$ and all sinks have indegree $1$, define
the concatenation $E \circ F$ of $E$ and $F$
as follows.  For $i = 1,\dotsc, n$, do
\begin{enumerate}
\item remove sink $i$ of $E$ and source $i$ of $F$,
\item merge each edge $(\vertex x, \text{sink }i)$ in $E$ with each edge
  $( \text{source }i, \vertex y )$ in $F$
  %to vertex $y$ 
  %are merged
  to form a single edge $(\vertex x, \vertex y)$ in $E \circ F$.
\end{enumerate}
Thus a concatenation of the form
%\begin{equation}\label{eq:basicconcat}
$F_{[a_1,b_1]} \circ \cdots \circ F_{[a_m,b_m]} \in \net n$
%\end{equation}
%in $\net n$ 
has $2n + m$ vertices:
$n$ sources inherited from $F_{[a_1,b_1]}$, $n$ sinks inherited from $F_{[a_m,b_m]}$, and $m$ internal vertices %which we will call 
$x_1,\dotsc, x_m$, where $x_j$ is inherited %comes 
from $F_{[a_j,b_j]}$.
Sometimes in a concatenation $E \circ F$,
%may be a multidigraph, because
there may exist internal vertices $\vertex x$ in $E$, $\vertex y$ in $F$
with
%such that a collection of
$\mu(\vertex x, \vertex y) > 1$ multiplicity-$1$ edges
%at least two edges are
incident upon both.
%Given
Define the {\em condensed concatenation}
$E \bullet F$ to be the
%simple
subdigraph of $E \circ F$ obtained
by removing, for all such pairs $(\vertex x, \vertex y)$,
%with $x \in G$, $y \in H$,
all but one of the $\mu(\vertex x, \vertex y)$ edges incident upon both,
%$$x$ and $y$,
and by marking this remaining edge with the multiplicity $\mu(\vertex x, \vertex y)$.
%Call
%$G \bullet H$ the {\em condensed concatenation} of $G$ and $H$.
%any $m-1$ of these edges.
For example, in $\net 4$ we have the 
%isomorphic 
graphs
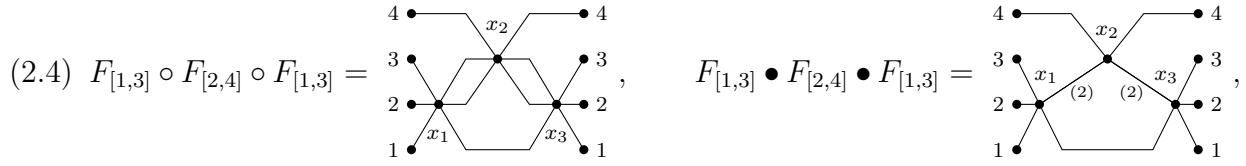
\begin{equation}\label{eq:allbutone2}
  F_{[1,3]} \circ F_{[2,4]} \circ F_{[1,3]} =
\begin{tikzpicture}[scale=.6,baseline=0]
\node at (-.8,1.5) {$\scriptstyle 4$};
\node at (-.8,0.5) {$\scriptstyle 3$};
\node at (-.8,-0.5) {$\scriptstyle 2$};
\node at (-.8,-1.5) {$\scriptstyle 1$};  
\node at (3.8,1.5) {$\scriptstyle 4$};
\node at (3.8,0.5) {$\scriptstyle 3$};
\node at (3.8,-0.5) {$\scriptstyle 2$};
\node at (3.8,-1.5) {$\scriptstyle 1$};  
\draw[-] (-.4,1.5) -- (.8,1.5) -- (2.2,-0.5) -- (3.4,-0.5);
\draw[-] (-.4,0.5) -- (.8,-1.5) -- (2.2,-1.5) -- (3.4,0.5);
\draw[-] (-.4,-0.5) -- (.8,-0.5) -- (2.2,1.5) -- (3.4,1.5);
\draw[-] (-.4,-1.5) -- (.8,0.5) -- (2.2,0.5) -- (3.4,-1.5);
\fill (-.4,1.5) circle  (1mm); \fill (-.4,.5) circle  (1mm); \fill (-.4,-.5) circle  (1mm); \fill (-.4,-1.5) circle  (1mm);
\fill (3.4,1.5) circle  (1mm); \fill (3.4,.5) circle  (1mm); \fill (3.4,-.5) circle  (1mm); \fill (3.4,-1.5) circle  (1mm);
\fill (.2, -.5) circle (1mm); \fill (1.5, .5) circle (1mm); \fill (2.8, -.5) circle (1mm);
\node at (0.2, -1.2) {$\scriptstyle x_1$};
\node at (1.5, 1.2) {$\scriptstyle x_2$};
\node at (2.8, -1.2) {$\scriptstyle x_3$};
\end{tikzpicture},
\qquad
%\; \not \cong \;
  F_{[1,3]} \bullet F_{[2,4]} \bullet F_{[1,3]} =
\begin{tikzpicture}[scale=.6,baseline=0]
\node at (-0.9,1.5) {$\scriptstyle 4$};
\node at (-0.9,0.5) {$\scriptstyle 3$};
\node at (-0.9,-0.5) {$\scriptstyle{2}$};
\node at (-0.9,-1.5) {$\scriptstyle{1}$};  
\node at (2.05,-.25) {$\scriptstyle _{(2)}$};
%\node at (2.15,-.35) {$\scriptstyle _{(2)}$};
\node at (.95,-.25) {$\scriptstyle _{(2)}$};
%\node at (.85,-.35) {$\scriptstyle _{(2)}$};
\node at (3.9,1.5) {$\scriptstyle 4$};
\node at (3.9,0.5) {$\scriptstyle 3$};
\node at (3.9,-0.5) {$\scriptstyle{2}$};
\node at (3.9,-1.5) {$\scriptstyle{1}$};  
\draw[-] (-0.5,1.5) -- (.7,1.5) -- (1.5,0.5) -- (2.3,1.5) -- (3.5,1.5);
\draw[-] (-0.5,0.5) -- (-0,-0.5) -- (1.5,0.5) -- (3,-0.5) -- (3.5,0.5);
\draw[-] (-0.5,-0.5) -- (-0,-0.5) -- (1.5,0.5) -- (3,-0.5) -- (3.5,-0.5);
\draw[-] (-0.5,-1.5) -- (-0,-0.5) -- (.5,-1.5) -- (2.5,-1.5) -- (3,-0.5) -- (3.5,-1.5);
\fill (-.5,1.5) circle  (1mm); \fill (-.5,.5) circle  (1mm); \fill (-.5,-.5) circle  (1mm); \fill (-.5,-1.5) circle  (1mm);
\fill (3.5,1.5) circle  (1mm); \fill (3.5,.5) circle  (1mm); \fill (3.5,-.5) circle  (1mm); \fill (3.5,-1.5) circle  (1mm);
\fill (0, -.5) circle (1mm); 
\fill (1.5, .5) circle (1mm); \fill (3, -.5) circle (1mm);
\node at (0.15, 0.1) {$\scriptstyle x_1$};
\node at (1.5, 1.1) {$\scriptstyle x_2$};
\node at (2.8, 0.1) {$\scriptstyle x_3$};
\end{tikzpicture},
\end{equation}
in which the two multiplicity-$2$ edges $(x_1,x_2)$,
$(x_2,x_3)$ of 
$F_{[1,3]}\bullet F_{[2,4]} \bullet F_{[1,3]}$
are the remnants of pairs of edges incident upon the same 
internal vertices in $F_{[1,3]} \circ F_{[2,4]} \circ F_{[1,3]}$.

Define a {\em star network} to be an element 
\begin{equation*}
    F_{[a_1,b_1]} \cdots F_{[a_m,b_m]}
\end{equation*}
of $\net n$ constructed by any combination of concatenation and condensed concatenation of finitely many simple star networks.  Regardless of the types of concatenation used, let $x_1,\dotsc,x_m$
denote the central vertices inherited from the $m$ simple star networks, and define
numbers
\begin{equation}\label{eq:muxixj}
\mu(x_i,x_j) = \# \text{ edges incident upon both $x_i$ and $x_j$ in } F_{[a_1,b_1]} \circ \cdots \circ F_{[a_m,b_m]}.
\end{equation}
Let $\ocnet n$ be the set of all star networks of order $n$. 
%and let $\cnet n$ denote the subset of these constructed by condensed concatenation only. 
%($\star$ Is it necessary to introduce the notation $\cnet n$?)
% consisting of condensed concatenations
% of finitely many simple star networks. 
For $F \in \ocnet n$, call a sequence $\pi = (\pi_1,\dotsc,\pi_n)$
of source-to-sink paths in 
%a star network $F \in \ocnet n$
$F$ a {\em path family of type $v = v_1 \cdots v_n \in \sn$}
 if for all $i$, path $\pi_i$ begins at source $i$ and terminates at sink $v_i$.  Say that $\pi$ {\em covers} $F$ if each multiplicity-$p$ edge $f$ in $F$ 
 %of multiplicity $p$ 
 belongs to
 %$(x_i,x_j)$ of multiplicity $\mu(x_i,x_j)$ in $F$ belongs to $\mu(x_i,x_j)$ 
 $p$ of the paths in $\pi$, 
 and define the sets
%For $F \in \net n$ and $u \in \sn$, define the sets
\begin{equation}\label{eq:pathfams}
  \begin{gathered}
    \Pi(F) = \{ \pi \,|\, \pi \text{ a path family covering } F \},\\
    \Pi_v(F) = \{ \pi \in \Pi(F) \,|\, \type(\pi) = v \}.
%    \Pi_{u,d}(F) = \{ \pi \in \Pi_u(F) \,|\, \dfct(\pi) = d \}.
    \end{gathered}
\end{equation}

In terms of the definitions \eqref{eq:pathfams}, we may combinatorially interpret products of %reversals in $\sn$ or of the related 
elements (\ref{eq:Dab}) quite simply.  
We say that $F$ {\em graphically represents}
\begin{equation}\label{eq:PiwF}
    \sum_{v \in \sn} %\sum_{\pi \in \Pi_w(F)} 
    \ntksp |\Pi_v(F)|\,v
\end{equation}
{\em as an element of $\zsn$}.
For
$F = F_{[a_1,b_1]} \circ \cdots \circ F_{[a_m,b_m]}$,
this element is
$D_{[a_1,b_1]} \cdots D_{[a_m,b_m]}$;
for
$F = F_{[a_1,b_1]} \bullet \cdots \bullet F_{[a_m,b_m]}$,
it is
%this element is
$D_{[a_1,b_1]} \cdots D_{[a_m,b_m]}$
divided by the product, over all pairs $i < j$,
%edges $(x_i,x_j)$,  
of the numbers $\mu(x_i,x_j)!$ defined in \eqref{eq:muxixj}.
In either case, we
%it is easy to 
see that the coefficient of $e$ in \eqref{eq:PiwF} is at least $1$: the unique noncrossing path family covering $F$ has type $e$.

\section{The Hecke algebra and planar networks}\label{s:hnqplanar}

Define the {\em (type-$A$ Iwahori-) Hecke algebra} $\hnq$
to be the
$\zqq$-span of its {\em natural basis} $\{ T_w \,|\, w \in \sn \}$,
with multiplication given by
\begin{equation*}
    T_{s_i} T_w = \begin{cases}
        T_{s_i w} &\text{if $s_i w > w$},\\
        (q-1)T_{s_i w} + q T_w &\text{if $s_i w < w$}.
        \end{cases}
\end{equation*}
%algebra with
% multiplicative identity element
% %s $e$ and 
% $T_e$,
% %respectively,
% generated over %$\mathbb Z$ and 
% $\zqq$
% by elements
% %$s_1,\dotsc, s_{n-1}$ and 
% $T_{s_1},\dotsc, T_{s_{n-1}}$, 
% subject to the relations
% \begin{equation}\label{eq:hnqdef}
% \begin{aligned}
% %\begin{alignedat}{3}
% %s_i^2 &= e &\qquad
% T_{s_i}^2 &= (q-1) T_{s_i} + qT_e &\qquad
% &\text{for $i = 1, \dotsc, n-1$},\\
% %s_is_js_i &= s_js_is_j &\qquad
% T_{s_i}T_{s_j}T_{s_i} &= T_{s_j}T_{s_i}T_{s_j} &\qquad
% &\text{for $|i - j| = 1$},\\
% %s_is_j &= s_js_i &\qquad
% T_{s_i}T_{s_j} &= T_{s_j}T_{s_i} &\qquad
% &\text{for $|i - j| \geq 2$}.
% %\end{alignedat}
% \end{aligned}\end{equation}
%Analogous to the natural basis $\{ w \,|\, w \in \sn \}$ of $\zsn$
%is the natural basis $\{ T_w \,|\, w \in \sn \}$ of $\hnq$,
%where we define
%$T_w = T_{s_{i_1}} \ntksp \cdots T_{s_{i_\ell}}$
%whenever $\sprod i\ell$
%is a reduced %(short as possible)
%expression for $w$ in $\sn$.  %We call $\ell$ the {\em length} of $w$ and write $\ell = \ell(w)$.
%We define the {\em one-line notation} $w_1 \cdots w_n$ of $w \in \sn$ by letting any expression for $w$ act on the word $1 \cdots n$,
%where each generator $s_j = s_{[j,j+1]}$ acts on an $n$-letter word by swapping the letters in positions $j$ and $j+1$, i.e., $s_j \circ v_1 \cdots v_n = v_1 \cdots v_{j-1} v_{j+1} v_j v_{j+2} \cdots v_n$.
%It is easy to show that $\ell(w)$ is equal to $\inv(w)$, the number of inversions in the one-line notation $w_1 \cdots w_n$ of $w$.
%($\star$ Define pattern avoidance.)
Specializing at $\qp12 = 1$ we have $H_n(1) \cong \zsn$
with $T_w \mapsto w$.

A semilinear involution on $\hnq$, known as the {\em bar involution}, is defined by
\begin{equation*}
%\begin{gathered}
      \ol{\qp12} \defeq \qm12, \qquad \ol{T_w} \defeq (T_{w^{-1}})^{-1}, \qquad
%  \sum_{w \in \sn} A_w(q) T_w \mapsto 
\overline{\sum_{w \in \sn} B_w(q) T_w}
  \defeq \sum_{w \in \sn} \ol{B_w(q)}\, \ol{T_w}.
%  \end{gathered}
\end{equation*}
%is defined by
%\begin{equation*}
%\end{equation*}
Kazhdan and Lusztig showed~\cite{KLRepCH} that $\hnq$ has a unique
%bar-invariant
basis $\{ C'_w \,|\, w \in \sn \}$
satisfying 
%\begin{enumerate}
%    \item 
%\end{enumerate}($\star$ or maybe say this more kindly?)
%\begin{equation*}
$\ol{C'_w} = C'_w$ for all $w$
% \qquad
and
%\item 
\begin{equation}\label{eq:klbasis}
\qp{\ell(w)}2 C'_w = \sum_{v \leq w} P_{v,w}(q) T_w,
\end{equation}
where coefficients $P_{v,w}(q) \in \mathbb N[q]$, known as the {\em Kazhdan--Lusztig polynomials}, have constant term $1$, satisfy
$\deg(P_{v,w}(q)) < \frac{\ell(w) - \ell(v)-1}2$ for $v < w$, 
and satisfy
$P_{w,w}(q) = 1$ for all $w$.
%\begin{enumerate}
%    \item $P_{w,w}(q) = 1$ for all $w$,
%   \item %$P_{v,w}(q) = 0$ for $v \not \leq w$
%$\deg(P_{v,w}(q)) < \frac{\ell(w) - \ell(v)-1}2$ for $v < w$.
%It is known that these polynomials
%the {\em Kazhdan--Lusztig polynomials} in (\ref{eq:klbasis}) 
%satisfy $P_{v,w}(q) \in \mathbb N[q]$, and $P_{v,w}(0) = 1$ for $v \leq w$.  
We also have that if
%Also~\cite{LakSan}, if
%the one-line notation for $w$ contains no subword matching the pattern
$w$ \avoidsp,
%\avoidsp, 
then $P_{v,w}(q) = 1$ for all $v \leq w$~\cite{LakSan}.
%$3412$ or $4231$ (e.g., if $w$ is a reversal),
%then the Kazhdan-Lusztig polynomials in (\ref{eq:klbasis})
%are identically $1$~\cite{LakSan}.
For convenience, we define 
\begin{equation}\label{eq: change of basis}
\wtc wq \defeq \qp{\ell(w)}2 C'_w
\end{equation}
and work with this (modified) 
Kazhdan--Lusztig basis $\{ \wtc wq \,|\, w \in \sn \}$.  These elements and their products
appear in various settings, including
%many areas of mathematics including
%of these elements expand 
%roducts of these basis elements and their expansions 
%nonnegatively in the natural and Kazhdan--Lusztig bases,
%and 
%the resulting coefficientsW
%have appeared in 
intersection homology~\cite{BBDFaisceaux}, \cite{SpringerQACI}, 
algorithmic and combinatorial description of Kazhdan--Lusztig basis elements themselves~\cite{BWHex}, \cite{Deodhar90}, 
Schubert varieties~\cite{BWHex}, total nonnegativity~\cite{GJImm}, \cite{SkanNNDCB}, \cite{StemImm}, \cite{StemConj}, 
trace evaluations~\cite{CHSSkanEKL}, \cite{CSkanTNNChar}, \cite{GreeneImm}, \cite{KLSBasesQMBIndSgn}, \cite{SkanHyperGC}, 
and symmetric functions~\cite{CHSSkanEKL}, \cite{HaimanHecke}, \cite{SkanHyperGC}. 
Each element $D_{[a,b]}$ in Section \ref{s:snplanar} is the $\qp 12 = 1$ specialization of $\wtc{s_{[a,b]}}q$.

%We will focus on Deodhar's result
%work in \cite{Deodhar90}.
%($\star$ Motivate the study of products of Kazhdan--Lusztig basis elements with examples from the literature.)
%In 
Deodhar~\cite[Prop.\,3.5]{Deodhar90}
%concerning
%, where he
%considered 
studied sequences
%\begin{equation}\label{eq:genseq}
%\mathbf{s} = 
$(s_{i_1},\dotsc,s_{i_m})$
%\end{equation}
of generators of $\sn$, 
%in $S$, 
products of the 
corresponding 
Kazhdan--Lusztig basis elements 
$\wtc{\smash{s_{i_j}}}q = T_e + T_{\smash{s_{i_j}}}$ of $\hnq$, 
%of 
%a
%the Kazhdan--Lusztig basis of 
%$H$,
%certain modified signless {\em Kazhdan--Lusztig} basis $\{ \wtc wq \,|\, w \in W \}$ of $H$~\cite{KLRepCH},
and their natural expansions
%of the 
%corresponding 
%resulting products
\begin{equation}\label{eq:Deoprod}
\wtc{s_{i_1}\ntksp}q \cdots \wtc{s_{i_m}\ntksp}q
%    (T_e + T_{s_{i_1}}) \cdots (T_e + T_{s_{i_k}}) 
= \sum_{v \in \sn} A_v(q) T_v.
\end{equation}
%in $H$.  
%Each factor on the left-hand side of (\ref{eq:Deoprod}) is a multiple of an element of the signless {\em Kazhdan--Lusztig basis}~\cite{KLRepCH} of $\hnq$. 
%Deodhar
He 
described 
the 
resulting 
coefficients $\{ A_v(q) \,|\, v \in \sn \} \subset \mathbb N[q]$ 
%and Deodhar described these
%\subset \mathbb Z[q]$ 
in terms of {\em subexpressions}
%(essentially) as follows.
%(with different notation).
%(Our notation differs a bit from his.)
%Given sequence $\mathbf{s} = (s_{i_1}, \dotsc, s_{i_k})$ of generators in $S$, 
%Define a {\em subexpression} %$\mathbf{\sigma}$ 
of $(s_{i_1}, \dotsc, s_{i_m})$,
%Specifically, 
%the generator sequence $\mathbf s$ (\ref{eq:genseq})
%to be a 
sequences $\sigma = (
%\sigma_0, 
\sigma_1, \dotsc, \sigma_m)$ %of elements of $H$ satisfying
%\begin{enumerate}
%    \item 
with $\sigma_j \in \{ e, s_{i_j}\}$
%$\sigma_0 = e$ and
% $\sigma_j \in \{ \sigma_{j-1}, \sigma_{j-1}s_{i_j}\}$ 
for $j = 1,\dotsc, m$.
%\end{enumerate}
(Our treatment here differs slightly from that of \cite{Deodhar90} but is equivalent.) Call index $j$ a {\em defect} of $\sigma$ if 
\begin{equation}\label{eq:deodhardefect}
\sigma_1 \cdots \sigma_{j-1}s_{i_j} < \sigma_1 \cdots \sigma_{j-1}
\end{equation}
%$\sigma_{j-1}s_{i_j} < \sigma_{j-1}$ 
%(whether or not $\sigma_j = \sigma_{j-1}s_j$), 
and let $\dfct(\sigma)$ denote the number of defects of $\sigma$. 
(Observe that $j=1$ cannot be a defect: we have $s_{i_1} > e$ always.) 
Each coefficient on the right-hand side of (\ref{eq:Deoprod}) is given by
\begin{equation}\label{eq:Deodefectformula}
    A_v(q) = \sum_\sigma q^{\dfct(\sigma)},
\end{equation}
where the sum is over all subexpressions
$\sigma$ of 
%$\mathbf{s}$
$(s_{i_1},\dotsc, s_{i_m})$
satisfying $\sigma_1 \cdots \sigma_m = v$.
%$\sigma_k = w$.
%(Our terminology differs a bit from that of Deodhar.)
% Deodhar~\cite[Prop.\,3.5]{Deodhar90}
% considered sequences
% $(s_{i_1},\dotsc,s_{i_k})$
% of generators in $S$, 
% elements $\wtc{s_{i_j}\ntksp}q \defeq T_e + T_{s_{i_j}}$ of a
% certain modified signless {\em Kazhdan--Lusztig} basis $\{ \wtc wq \,|\, w \in W \}$ of $H$~\cite{KLRepCH},
% and
% natural expansions of the corresponding products
% \begin{equation}\label{eq:Deoprod}
% \wtc{s_{i_1}\ntksp}q \cdots \wtc{s_{i_k}\ntksp}q
% = \sum_{w \in W} A_w(q) T_w
% \end{equation}
% in $H$.  
% He described the 
% coefficients $\{ A_w(q) \,|\, w \in W \} \subset \mathbb Z[q]$ as follows.
% Define a {\em subexpression}
% of $(s_{i_1}, \dotsc, s_{i_k})$
% to be a sequence $\sigma = (\sigma_0, \sigma_1, \dotsc, \sigma_k)$ of elements of $H$ satisfying
% $\sigma_0 = e$ and
%  $\sigma_j \in \{ \sigma_{j-1}, \sigma_{j-1}s_{i_j}\}$ for $j = 1,\dotsc, k$.
% Call index $j$ a {\em defect} of $\sigma$ if $\sigma_{j-1}s_{i_j} < \sigma_{j-1}$ 
% and let $\dfct(\sigma)$ denote the number of defects of $\sigma$. Then each coefficient on the right-hand side of (\ref{eq:Deoprod}) is given by
% \begin{equation}\label{eq:Deodefectformula}
%     A_w(q) = \sum_\sigma q^{\dfct(\sigma)},
% \end{equation}
% where the sum is over all subexpressions
% $\sigma$ of 
% $(s_{i_1},\dotsc, s_{i_k})$
% satisfying $\sigma_k = w$.

%($\star$ Maybe use a symbol other than $\sigma$ above, in case we want to later use $\sigma$ as a path family of type $s$ for some generator $s$.  Or will these be so far apart that it doesn't matter?)

Billey and Warrington observed~\cite[Rmk.\,6]{BWHex}
that 
%when $W$ and $H$ are
%the symmetric group $\sn$ and type-$\msfA$ Iwahori--Hecke algebra
%$\hanq$, 
% the defect statistic has the following simple graphical interpretation.
the defect statistic has a simple graphical interpretation.
Specifically, subexpressions 
%\sigma = (\sigma_1,\dotsc,\sigma_k)$ 
of $(s_{i_1},\dotsc,s_{i_m})$ correspond bijectively to path families
%$\pi = (\pi_1,\dotsc,\pi_n)$ 
covering 
\begin{equation}\label{eq:wirediagprod}
F =
F_{[i_1,i_1+1]} \circ \cdots \circ F_{[i_m,i_m+1]} 
\end{equation}
in $\net n$ with 
%$\sigma = 
$(\sigma_1,\dotsc,\sigma_m)$ corresponding 
to the family $\pi \in \Pi(F)$
%$ = (\pi_1,\dotsc,\pi_n)$ 
constructed by prescribing %having  
\begin{equation*}
\text{the paths meeting at $\ctr j$ }
\begin{cases}
    \text{cross there} &\text{if $\sigma_j = s_{i_j}$},\\
    \text{do not cross there} &\text{if $\sigma_j = e$}.
    \end{cases}
\end{equation*}
By this bijection, index $j$ is a defect of $\sigma$ in the sense of (\ref{eq:deodhardefect}) if and only if the paths meeting at $\ctr j$
have previously crossed an odd number of times.
%at $\ctr 1, \dotsc, \ctr{j-1}$.
%$j$ is a defect of $\pi$ in the sense of Definition~\ref{d:wddefect}.
%($\star$ Define defect of a path family.)

Clearwater--Skandera 
%and the second author 
extended this 
result~\cite[Cor.\,5.3]{CSkanTNNChar} to products of the form
\begin{equation}\label{eq:CSprod}
%\wtc{v^{(1)}}q \cdots \wtc{v^{(k)}}q
\wtc{s_{[a_1,b_1]}}q \cdots \wtc{s_{[a_m,b_m]}}q
%\bigg( \sum_{u \leq v^{(1)}} T_u \bigg) 
%\cdots
%\bigg( \sum_{u \leq v^{(k)}} T_u \bigg)
= \sum_{v \in \sn} A_v(q) T_v,
\end{equation}
%in $\hnq$,
where 
%$v^{(1)}, \dotsc, v^{(k)}$
%are maximal elements of parabolic subgroups of $\sn$, and 
each factor satisfies
\begin{equation*}
    \wtc{s_{[a_j,b_j]}}q = \nTksp \sum_{u \leq s_{[a_j,b_j]}} \nTksp T_u,
\end{equation*}
%are more general elements of the modified signless Kazhdan--Lusztig basis of $\hnq$
%~\cite[Cor.\,5.3]{CSkanTNNChar}.
%Again, these factors are multiples of the signless Kazhdan--Lusztig
%basis of $\hnq$.
%the corresponding 
since reversals \avoidp.
%(See \cite{HPSWDefectBC} for a type-$\msfBC$ extension.)
%belongs to the modified signless Kazhdan--Lusztig basis of $\hanq$.
This extension requires a more general definition of defects.
While the intersection of two paths in 
\eqref{eq:wirediagprod} is a union of vertices,
the intersection of two paths in a more general star network
\begin{equation*}
%\label{eq:bulletconcat}
F = F_{[a_1,b_1]} \cdots F_{[a_m,b_m]}
\end{equation*}
is a subgraph of $F$ whose connected components are vertices or are paths of the form 
\begin{equation}\label{eq:component}
(x_k,\dotsc, x_\ell)
\end{equation}
for some $k < \ell$.  (For example, consider the unique noncrossing path family covering the second network in \eqref{eq:allbutone2}.) %($\star$ Include source and sink vertices?) 
For each initial vertex $x_k$ in a component \eqref{eq:component} of the intersection of two paths, we will say that the paths {\em meet} at $x_k$.
%Let us say that two paths {\em meet} at a vertex $x_k$ if they enter $x_k$ 
Our embedding of star networks
in the plane naturally allows us to declare an edge entering (exiting) a vertex $x_k$ to be above or below another edge entering (exiting) $x_k$.
%(\ref{eq:})
%We will call a component 
We will call a component \eqref{eq:component} in the intersection of paths $\pi_i$, $\pi_j$,
a {\em crossing} of $\pi_i$ and $\pi_j$ if the two paths enter $x_k$ and exit $x_\ell$ in different orders.
Extending the Billey--Warrington definition of defect to accommodate three or more paths passing through a vertex, we have the following~\cite[\S 5]{CSkanTNNChar}.
%above $\pi_j$ (or vice versa).
%Again we have (\ref{eq:BWdefectformula}), where 
%$F$ has the form
%$F_{v^{(1)}} \circ \cdots \circ F_{v^{(k)}}$, 
%with factors
%belonging to a class of planar networks
%%generalizing those in
%(\ref{eq:A generator wirings}) in that their
%by allowing 
%interior vertices may have indegree and outdegree greater than $2$.

\begin{defn}\label{d:defect}
    Given star network  $F \in \ocnet{n}$ having internal vertices $x_1, x_2, \dotsc,$ and 
    path family $\pi$ covering $F$, 
%    $F \in \ocnet{n}$,
    % = F_{[a_1,b_1]} \bullet \cdots \bullet F_{[a_m,b_m]}
define a {\em defect of $\pi$} at 
%the $k$th star 
$\ctr k$ to be
a triple $(\pi_i, \pi_j, k)$ with
%\begin{enumerate}
%\item $i \neq j$,
%\item 
$i < j$
%\item $0 < j$,
%\item 
and $\pi_i$ and $\pi_j$ meeting at $\ctr k$ after having crossed an odd number of times. 
Define $\dfct(\pi)$ to be the number of defects of $\pi$.
%\end{enumerate}

%($\star$ edit this definition of \emph{defect} to include the condensed case?)
\end{defn}

\begin{obs}\label{o:defect}
    A defect $(\pi_i, \pi_j, k)$ of a path family $\pi$ must satisfy $k \geq 2$.
\end{obs}

We may state an alternative formulation of the defect statistic in terms of the natural vertical ordering of edges entering an internal vertex $x_k$ of a planar network.
%paths in a star network. 
%The natural vertical ordering of edges entering a vertex $x_k$ 
This ordering induces relations $\prec_k$, $\sim_k$, $\precsim_k$ on paths which pass through $x_k$.  While these relations depend upon the index $k$, we economize notation by omitting $k$ when the vertex $x_k$ is clear from context.
We declare that $\pi_i\prec \pi_j$ if $\pi_i$ enters $x_k$ below $\pi_j$, 
that $\pi_i\sim\pi_j$ 
%if $\pi_i\precsim \pi_j$ and $\pi_j\precsim \pi_i$ (equivalently 
if $\pi_i, \pi_j$ enter $x_k$ on the same edge,
and that $\pi_i\precsim\pi_j$ if 
$\pi_i \prec \pi_j$ or $\pi_i \sim \pi_j$.
%$\pi_i$ enters $x_k$ below or on the same edge as $\pi_j$, and 
 For example, consider the path family
% ($\star$ add a figure: $\Pi(F_{[2,4]}\bullet F_{[1,3]})$ (ascending))
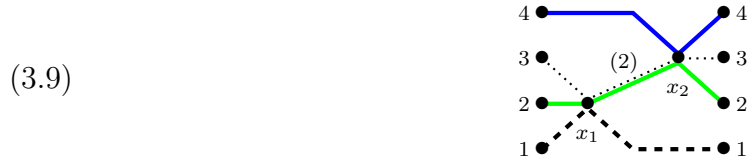
\begin{equation}
    \begin{tikzpicture}[scale=.6,baseline=40]
      % Define variables
      \pgfmathsetmacro{\k}{4}  % Set k = 4
      \pgfmathsetmacro{\n}{4}  % Set n = 7
      \pgfmathsetmacro{\indexOffset}{0.4}
      % Define x-coordinates
      \pgfmathsetmacro{\xstart}{0}
      \pgfmathsetmacro{\xstep}{0.5}
    
      % Draw indices on left and right
      \foreach \index in {1, ..., \n} {
        \node at (\xstart - \indexOffset, \index) {$\scriptstyle \index$};
        \node at (\xstart + \xstep*2*\k + \indexOffset, \index) {$\scriptstyle \index$};
      }
      \node at (1.8, 2.9) {$\scriptstyle (2)$};
      \node at (1, 1.3) {$\scriptstyle x_1$};
      \node at (3, 2.3) {$\scriptstyle x_2$};

    \draw[-, ultra thick, dashed]
        (0,1) -- (1, 1.9) -- (2,1) -- (4,1);
    
        \draw[-, ultra thick, color=green]
        (0,2) -- (1, 2) -- (3, 2.9) -- (4,2);
    
        \draw[-, thick, dotted]
        (0,3) -- (1, 2.1) -- (3, 3) -- (4,3);
    
        \draw[-, ultra thick, color=blue]
        (0,4) -- (2, 4) -- (3, 3.1) -- (4,4);
        
        \node at (0, 1) {$\bullet$};
        \node at (0, 2) {$\bullet$};
        \node at (0, 3) {$\bullet$};
        \node at (0, 4) {$\bullet$};
        \node at (4, 1) {$\bullet$};
        \node at (4, 2) {$\bullet$};
        \node at (4, 3) {$\bullet$};
        \node at (4, 4) {$\bullet$};
        \node at (1, 2) {$\bullet$};
        \node at (3, 3) {$\bullet$};
    \end{tikzpicture}
\end{equation}
and the 
%order in which 
paths entering the vertex $x_2$.
%Defining $\precsim$ 
In terms of this vertex we have
%$x_2$, we have 
\begin{equation*}
\pi_2\sim\pi_3,
\quad  
\pi_2\prec \pi_4,
\quad 
\pi_3\prec \pi_4.
\end{equation*}
%(Also, $\pi_2\precsim\pi_4$ and $\pi_2\precsim\pi_3$.)
\begin{obs}
    The relations $\prec$, $\sim$, $\precsim$ defined in terms of any fixed internal vertex of a star network $F$
    %, $\sim$ 
    are a (strict) partial order, an equivalence relation, and a preorder, respectively, on source-to-sink paths in $F$.
    %transitive, ($\star$ and what other properties?) and 
    These relations satisfy
    \begin{enumerate}
        \item $\pi_i \prec \pi_j \sim \pi_\ell$ implies $\pi_i \prec \pi_\ell$, 
        \item $\pi_i \sim \pi_j \prec \pi_\ell$ implies $\pi_i \prec \pi_\ell$,
        \item $\pi_i \not \prec \pi_j$ is equivalent to $\pi_j \precsim \pi_i$.
%        \item $\star$ What else?
    \end{enumerate}
\end{obs}
In terms of the partial order $\prec = \prec_k$, we can define a defect at vertex $x_k$ as follows.
%ing way.
\begin{lem}\label{l: alt defect defn}
    For star network $F \in \ocnet n$ and path family $\pi \in \Pi(F)$, the triple $(\pi_i, \pi_j, k)$ is a defect if and only if $i < j$ and $\pi_j \prec \pi_i$.
    %enters $F_{[a_k, b_k]}$ above $\pi_j$ with $i<j$.
    % For each triple $(\pi_i, \pi_j, k)$ with $i<j$ and 
    % %covering $F$ and a triple $(\pi_i, \pi_j, k)$ with $i<j$ and 
    % $\pi_i$ and $\pi_j$ meeting at $\ctr k$,
    % %the internal vertex of $F_{[a_k,b_k]}$, 
    % the triple is a defect if and only if
    % $\pi_i$ enters $F_{[a_k,b_k]}$ 
    % above
    % %at a source vertex of a greater index than that which 
    % $\pi_j$.
    % %enters.
\end{lem}
%F_{[a_1,b_1]} \bullet \cdots \bullet F_{[a_m,b_m]} 
\begin{proof}
    %Assume $i < j$.  Then 
    Write $F = F_{[a_1,b_1]} \cdots  F_{[a_m,b_m]}$. Assuming $\pi_i, \pi_j$ with $i < j$ meet at $\ctr k$, we have that $\pi_i$ enters $\ctr k$
    %$F_{[a_k,b_k]}$ 
    above $\pi_j$ if and only if the two paths have crossed an odd number of times in $F_{[a_1,b_1]} \cdots F_{[a_{k-1},b_{k-1}]}$.
    %then immediately after an odd number of crossings of $\pi_i$ and $\pi_j$, we have that $\pi_i$ is above $\pi_j$. 
    % If $\pi_i$ enters $F_{[a_k,b_k]}$ above $\pi_j$, then the paths must have crossed an odd number of times, so $(\pi_i, \pi_j, k)$ is a defect.
\end{proof}

Extending the algebraic interpretation \eqref{eq:PiwF} of a planar network, we define the set
\begin{equation}\label{eq:Pi with v and d}
    \Pi_{v,d}(F) = \{ \pi \in \Pi_v(F) \,|\, \dfct(\pi) = d \}
\end{equation}
and we
say that $F$ {\em graphically
represents}
%the element}
\begin{equation}\label{eq:Gtozhnq}
    \sum_{v \in \sn} \sum_{d \geq 0} |\Pi_{v,d}(F)|q^d T_v 
    = \nTksp\sum_{\pi \in \Pi(F)}\nTksp q^{\dfct(\pi)}T_{\type(\pi)}
\end{equation}
{\em as an element of $\hnq$}.
%By \cite[Cor.\,5.3]{CSkanTNNChar}, %\cite[Prop.\,4.2]{SkanNNDCB} we have that
For $F = F_{[a_1,b_1]} \circ \cdots \circ F_{[a_m,b_m]}$, this element is
\begin{equation}\label{eq:KLprod}
    \wtc{s_{[a_1,b_1]}}q \cdots \wtc{s_{[a_m,b_m]}}q
    \end{equation}
  \cite[Cor.\,5.3]{CSkanTNNChar}, \cite[Prop.\,4.2]{SkanNNDCB};  
for $F = F_{[a_1,b_1]} \bullet \cdots \bullet F_{[a_m,b_m]}$, it is
%$\wtc{s_{[a_1,b_1]}}q \cdots \wtc{s_{[a_m,b_m]}}q$, 
\eqref{eq:KLprod} divided by
the product, over all pairs %$(i,j)$, 
$i < j$, of the
%of the numbers $\mu(x_i,x_j)!$.
$q$-factorial polynomials 
$\mu(x_i,x_j)_q!$, 
%of the numbers 
 where we define $\mu(x_i,x_j)$ as in \eqref{eq:muxixj} and
\begin{equation*}
    [p]_q \defeq \begin{cases} 1 + q + \cdots + q^{p-1} &\text{if $p > 0$},\\
    0 &\text{if $p = 0$},
    \end{cases}
    \qquad\quad
    [p]_q! \defeq \begin{cases} [1]_q [2]_q \cdots [p]_q &\text{if $p > 0$},\\
    1 &\text{if $p = 0$}.
    \end{cases}
\end{equation*}
(See, e.g., \cite{StanEC1}.)
Summarizing results in \cite{CSkanTNNChar}, \cite{SkanNNDCB}, we prove the second fact as follows.

%\begin{lem}\label{l:qfact}
%    Let $F \in \ocnet n$ have a pair of internal vertices $\ctr i, \ctr j$ with $p$ edges incident upon both. Fix a path family $\pi \in \Pi(F)$ such that the $p$ paths traversing these edges, say $\pi_{u_1}, \dots, \pi_{u_p}$, satisfy 
%    \begin{equation*}
%        \pi_{u_1} \prec \cdots \prec \pi_{u_p}, \quad u_1 < \cdots < u_p,
%    \end{equation*}
%    and let $S$ be the set of all path families in $\Pi(F)$ agreeing with $\pi$ everywhere outside of these $p$ edges. Then we have
%    \begin{equation}\label{eq:preimage}
%        \sum_{\pi \in \theta^{-1}(\sigma)} q^{\dfct(\pi)} = [p]_q! q^{\dfct(\sigma)} .
%    \end{equation}
%\end{lem}
%\begin{proof}
%    ($\star$ Prove this)
%\end{proof}

% \begin{lem}
% Let $\pi \in \Pi_v(F)$ have 
% Let $F$ have $p$ edges
% from $\ctr i$ to $\ctr j$, and let $\pi \in \Pi_v(F)$ be a path family with paths
% $\pi_{u_1},\dotsc,\pi_{u_p}$ passing through
% $\ctr i$ and $\ctr j$, with the $\ctr i$-to-$\ctr j$ subpaths appearing from bottom to top in the
% order $\pi_{u_1}$ below $\pi_{u_2}$ below $\dotsc, \pi_{u_p}$.
% \end{lem}

%\begin{lem}
%    Let $e_1, \dots, e_p$ be the edges incident upon both $\ctr i$ and $\ctr j$, and suppose that these are replaced by a single multiplicity-$p$ edge $e$ in creating $F'$\ntnsp. Define a map $\theta : \Pi(F) \rightarrow \Pi(F')$ by replacing each edge $e_k$ in a path of $\pi \in \Pi(F)$ by $e$.
%\end{lem}

\begin{prop}\label{p:qfact}
    Fix $F \in \ocnet n$, assume that %pair of internal vertices 
    %$\ctr i$, $\ctr j$ 
    %are incident upon 
    %$m_{i, j}$ 
$p > 1$
multiplicity-$1$ edges are incident upon some pair of internal vertices, and construct $F' \in \ocnet n$ by replacing these edges by one edge of multiplicity $p$.
%$m_{i,j}$. 
Then $F'$ graphically represents the element
%    \begin{equation*}
%        \frac 1 {[m_{i, j}]_q!} \sum_{v \in \sn} \sum _{d \geq 0} |\Pi_{v, d}(F)| q^d T_v.
%    \end{equation*}
%    or
    \begin{equation*}
     \frac 1{[p]_q!} 
     %{[m_{i,j}]_q!} 
     \ntksp\sum_{\pi \in \Pi(F)}\ntksp q^{\dfct(\pi)}T_{\type(\pi)}.
     \end{equation*}
\end{prop}
\begin{proof}
    Let $\ctr r$, $\ctr s$ be the pair of internal vertices in $F$, let $e_1, \dots, e_p$ be the edges incident upon both vertices
    %$\ctr r$, $\ctr s$ 
    in $F$, and let $e$ be the multiplicity-$p$ edge replacing these in $F'$.
    %$\ctr i$ and $\ctr j$, and suppose that these are replaced by a single multiplicity-$p$ edge $e$ in creating $F'$. 
    %Observe that 
    Thus for every path family 
    %$\pi \in 
    in $\Pi(F)$,
    $p$ of its component paths contain the edges $e_1,\dotsc, e_p$,
    and for every path family in $\Pi(F')$,
    $p$ of its component paths contain the edge $e$.
    Define the map 
%    \begin{equation*}
%    \begin{aligned}
        $\theta : \Pi(F) \rightarrow \Pi(F')$
        %\\
%        \pi &\mapsto \pi'
%        \end{aligned}
%        \end{equation}
        by replacing edges $e_1,\dotsc, e_p$ in each path family $\pi \in \Pi(F)$ by $p$ copies of the edge $e$.
        %substituting $e$ for these $p$ edges in the $p$ paths.
        
%        and that these preimages partition $\Pi(F)$,
%        \begin{equation*}           
%        \end{equation*}
             It is clear that 
             the map $\theta$ is surjective and that 
             for each path family $\sigma \in \Pi(F')$, 
        %we have
        %\begin{enumerate}
        %that 
        the set $\theta^{-1}(\sigma) \subseteq \Pi(F)$ consists of $p!$ path families $\pi$ satisfying $\type(\pi) =  \type(\sigma)$.
        We claim that the defects of
            path families in each preimage $\theta^{-1}(\sigma)$ satisfy
    \begin{equation}\label{eq:preimage}
        \sum_{\pi \in \theta^{-1}(\sigma)} 
        \nTksp q^{\dfct(\pi)} = [p]_q! q^{\dfct(\sigma)} .
    \end{equation}
    To see this, fix $\pi \in \theta^{-1}(\sigma)$ and observe that for all internal vertices $\ctr t \neq \ctr s$ of $F$ and $F'$, the triple $(\pi_i, \pi_j,  t)$ is defective if and only if $(\sigma_i, \sigma_j, t)$ is defective:
    the agreement of $\pi$ and $\sigma$ outside of the subgraph induced by $\ctr r$, $\ctr s$ guarantees that
    the preorder $\prec_t$ defined in terms of any internal vertex $\ctr t$ satisfies
    $\pi_i \prec_t \pi_j$ if and only if $\sigma_i \prec_t \sigma_j$. It follows that the difference between $\dfct(\pi)$ and
    %is equal to 
    $\dfct(\sigma)$ is just the number of defects of $\pi$ at $\ctr s$.
    Letting $\pi_{u_1}, \dotsc, \pi_{u_p}$ be the $p$ paths passing through $\ctr r$, $\ctr s$, labeled by
%    But this is equal to the number of inversions in the word $u_1 \cdots u_p$ defined by containing $x$ and $y$, with
    %in terms of the preorder $\prec_y$ 
    %defined in terms of the internal vertex $y$ 
    \begin{equation*}
        \pi_{u_1} \prec_s \cdots \prec_s \pi_{u_p},
    \end{equation*}
%Thus 
we have that the number of defects of $\pi$ at $\ctr s$ 
equals the number of inversions in the word $u_1 \cdots u_p$.
Thus we have
\begin{equation*}
\sum_{\pi \in \theta^{-1}(\sigma)} \nTksp q^{\dfct(\pi)} = 
\sum_{u \in \mfs p} q^{\inv(u) + \dfct(\sigma)},
\end{equation*}
which implies \eqref{eq:preimage}.
%as desired.

%By definition and by \eqref{eq:preimage}, 
It follows that
$F'$ graphically represents the element
    \begin{equation}\label{eq:qfactsum}
        % \sum_{v \in \sn} \sum_{d \geq 0} |\Pi_{v,d}(F')|q^d T_v.
        \nTksp\sum_{\sigma \in \Pi(F')}\nTksp q^{\dfct(\sigma)}T_{\type(\sigma)}
    %\end{equation}
   % By \eqref{eq:preimage}, this is 
   % \begin{equation*}
        = \ntksp \sum_{\sigma \in \Pi(F')} 
        %\frac 1 {[m_{i,j}]_q!} 
        \frac 1{[p]_q!}
        \ntksp \sum_{\pi \in \theta^{-1}(\sigma)} \nTksp q^{\dfct(\pi)} T_{\type(\sigma)}.
        \end{equation}
        Since  $\type(\pi) = \type(\sigma)$
        for all $\pi \in \theta^{-1}(\sigma)$, and since
       \begin{equation*}
            \Pi(F) = \nTksp \bigcup_{\sigma \in \Pi(F')}\nTksp \theta^{-1}(\sigma),
        \end{equation*}
        we have the desired result.
%        It is clear that the map $\theta$ is surjective and that for each path family $\sigma \in \Pi(F')$, 
        %we have
        %\begin{enumerate}
        %that 
%        the set $\theta^{-1}(\sigma) \subseteq \Pi(F)$ consists of $p!$ path families $\pi$ satisfying $\type(\pi) =  \type(\sigma)$.
%Thus we have
%        \begin{equation*}
%            \Pi(F) = \nTksp \bigcup_{\sigma \in \Pi(F')}\nTksp \theta^{-1}(\sigma).
%            \end{equation*}
%            \begin{equation}
%                
%            \end{equation}
%            \item for all $\pi \in \theta^{-1}(\sigma)$, $\type(\pi) = \type(\sigma)$,
 %           \item $ \subset \Pi(F)$ contains $p!$ path families of type $\type(\sigma)$.
% Choose $v$ such that $\Pi_v(F)$ is nonempty.
% For each path family $\pi$ in $\Pi_{v}(F)$,
% the $m_{i,j}$ $x_i$-to-$x_j$ edges appear
% in $m_{i,j}$ different paths.
%     Define a map 
%     \begin{equation*}
%     \begin{aligned}
%         \Pi_{v}(F) &\rightarrow \Pi_{v}(F')\\
%         \pi &\mapsto \pi'
%         \end{aligned}
%         \end{equation*}
% by replacing the $m_{i,j}$ edges in the $m_{i,j}$ paths with
% %from $x_i$ to $x_j$ in $m_{i,j}$ paths of $\pi$ with 
% the single $x_i$-to-$x_j$ path in $F'$.
% subpaths of paths in $\pi$ 
%     with $m_{i,j}$ copies
%     $\pi = (\pi_1,\dotsc,\pi_n)$ be a path family covering $F$
\end{proof}

The product of polynomials $\mu(x_i,x_j)_q!$ 
%over all pairs $(i,j)$ 
mentioned after \eqref{eq:KLprod} also can be expressed in terms of intervals appearing in reversals~\cite[Thm.\,4.3]{SkanNNDCB}.
Given a sequence of $m$ intervals
 \begin{equation}\label{eq:intervalseq}
 \mathcal A = ([a_1,b_1], \dotsc, [a_m,b_m]),
 \end{equation}
 define $\tbinom m2$ more intervals $\{I_{i,j} \,|\, i < j \}$ by
 \begin{equation}\label{eq:intervaloverlap}
 I_{i,j} = [a_i,b_i] \cap [a_j,b_j] \ssm ( [a_{i+1},b_{i+1}] \cup \cdots \cup [a_{j-1},b_{j-1}]).
 \end{equation}
Let $f_{\ntnsp\mathcal A}(q)$ be the product of the $q$-factorials of the cardinalities of the intervals (\ref{eq:intervaloverlap}),
 %polynomial
 \begin{equation}\label{eq:overlappoly}
 f_{\ntnsp\mathcal A}(q) \defeq \prod_{i < j} |I_{i,j}|_q!.
 \end{equation}
% ($\star$ Mention rescaling $C'_w \mapsto \wtc wq$ somewhere. In \eqref{eq: change of basis}?)
Say that a Kazhdan--Lusztig basis element $\wtc wq$ has a {\em reversal factorization} %($\star$ or {\em parabolic factorization}?)
if there is a sequence (\ref{eq:intervalseq}) of intervals satisfying
\begin{equation}\label{eq:reversalfactor}
\wtc wq =
\frac1{f_{\ntnsp \mathcal A}(q)} \wtc{s_{[a_1,b_1]}}q \cdots \wtc{s_{[a_m,b_m]}}q.
\end{equation}
This property of a Kazhdan--Lusztig basis element $\wtc wq$ guarantees the following combinatorial interpretation of the Kazhdan--Lusztig polynomials $\{P_{v,w}(q) \mid v \leq w\}$~\cite[Cor.\,5.3]{CSkanTNNChar}.
\begin{prop}
    Let $\wtc wq$ have the reversal factorization \eqref{eq:reversalfactor} and define the planar network $F = F_{[a_1,b_1]} \bullet \cdots \bullet F_{[a_m,b_m]}$. 
    %as in \eqref{eq:bulletconcat}. 
    Then we have
    \begin{equation}
        P_{v,w}(q) = \sum_{d \geq 0} |\Pi_{v,d}(F)|q^d.
    \end{equation}
\end{prop}
\begin{proof}
    Let $G = F_{[a_1, b_1]} \circ \cdots \circ F_{[a_m, b_m]}$. By ~\cite[Cor.\,5.3]{CSkanTNNChar}, we have that the coefficient of $T_w$ in $\wtc {s_{[a_1, b_1]}}q \cdots \wtc {s_{[a_m, b_m]}}q$ is
    \begin{equation*}
        \sum_{d \geq 0} |\Pi_{v, d}(G)|q^d.
    \end{equation*}
    By Proposition \ref{p:qfact}, it follows that the network $F$
    %= F_{[a_1, b_1]} \bullet \cdots \bullet F_{[a_m, b_m]}$ 
    and polynomial $f_{\ntnsp\mathcal A}(q)$ \eqref{eq:overlappoly} satisfy 
    \begin{equation*}
        \sum_{d \geq 0} |\Pi_{v, d}(F)|q^d  = \frac 1 {f_{\ntnsp\mathcal A}(q)} \sum_{d \geq 0} |\Pi_{v, d}(G)|q^d.
    \end{equation*}
\end{proof}

Two results providing sufficient conditions on $w$ for a reversal factorization of $\wtc wq$ are the following~\cite[Thm.\,1]{BWHex}, \cite[Thm.\,4.3]{SkanNNDCB}.
\begin{thm}\label{t:BWHex}
If $w \in \sn$ avoids the patterns $321$, $56781234$, $46781235$, $56718234$, $46718235$,
and if $s_{i_1} \cdots s_{i_\ell}$ is a reduced expression for $w$, then we have
$\wtc wq = \wtc {s_{i_1}}q \cdots \wtc {s_{i_\ell}}q$.
\end{thm}
\begin{thm}\label{t:SkanSmooth}
    If $w \in \sn$ \avoidsp, then $\wtc wq$ has a reversal factorization.
    %there exists a sequence \eqref{eq:intervalseq} of intervals such that
    %we have the factorization \eqref{eq:reversalfactor}.
%    $\wtc wq = \wtc {s_{[a_1,b_1]}}q \cdots \wtc {s_{[a_\ell,b_\ell]}}q$.
\end{thm}
The combination of Theorems~\ref{t:BWHex}, \ref{t:SkanSmooth} is not the strongest factorization result possible.
%result on parabolic factorization of Kazhdan--Lusztig basis elements.
Indeed, the known reversal factorization $\wtc{4231}q = \wtc{s_{[1,2]}}q \wtc{s_{[2,4]}}q \wtc{s_{[1,2]}}q$ is guaranteed by neither theorem. We therefore have the following question \cite[Quest.\,4.5]{SkanNNDCB}.
\begin{quest}
    For which $w \in \sn$ does $\wtc wq$ have a reversal factorization?
\end{quest}

\section{Reduction of defects}\label{s:defectremoval} 

%($\star$ Add vertices to diagrams.)

Our defect reduction theorem, Theorem \ref{t:decrease defects by 1},
%main result about defects 
asserts that if a star network can be covered by a path family having $d$ defects, then it can also be covered by a path family of the same type having $d-1$ defects. In certain simple cases, we can easily produce a $(d-1)$-defect family in some set $\Pi_v(F)$ from a $d$-defect family in $\Pi_v(F)$ by swapping a pair of subpaths.
%map from $d$ to $d-1$ defects by changing a pair of paths that meet at an internal vertex and cross to meet but not cross, and changing behavior at previous meetings so that the type is maintained, thus 'undoing' a defect.
% ($\star$ give simple case where this does work, on $F_{s_1} \circ F_{s_1}$)
For example, consider the 
%In the following 
star network and path families
%$F_{[1,2]} \circ F_{[1,2]}$ and path families $\pi$ and $\sigma$ in $\Pi(F_{[1,2]} \circ F_{[1,2]})$,
%and path families, 
\begin{equation}\label{eq:2by2ex}
F_{[1,2]} \circ F_{[1,2]} = 
\begin{tikzpicture}[scale=.7,baseline=25]
  % Define variables
  \pgfmathsetmacro{\k}{2}  % Set k = 4
  \pgfmathsetmacro{\n}{2}  % Set n = 5
  \pgfmathsetmacro{\indexOffset}{0.4}
  % Define x-coordinates
  \pgfmathsetmacro{\xstart}{0}
  \pgfmathsetmacro{\xstep}{0.5}

  % Draw indices on left and right
  \foreach \index in {1, ..., \n} {
    \node at (\xstart - \indexOffset, \index) {$\scriptstyle \index$};
    \node at (\xstart + \xstep*2*\k + \indexOffset, \index) {$\scriptstyle \index$};
  }
  
 %F_[1,2] \circ F_[1,2] =    
   \draw[-,  thick]
    (0,1) -- (0.5,1.5) -- (1,1) -- (1.5,1.5) -- (2,1);
    \draw[-, thick]
    (0,2) -- (0.5,1.5) -- (1,2) -- (1.5,1.5) -- (2,2);
    \fill (.5, 1.5) circle (1mm);  \fill (1.5, 1.5) circle (1mm);
    \fill (0, 1) circle (1mm);  \fill (0, 2) circle (1mm);
    \fill (2, 1) circle (1mm);  \fill (2, 2) circle (1mm);
    \node at (0.55, 1) {$\scriptstyle x_1$};  \node at (1.55, 1) {$\scriptstyle x_2$};
\end{tikzpicture}, 
\qquad
\pi = 
\begin{tikzpicture}[scale=.7,baseline=25]
  % Define variables
  \pgfmathsetmacro{\k}{2}  % Set k = 4
  \pgfmathsetmacro{\n}{2}  % Set n = 5
  \pgfmathsetmacro{\indexOffset}{0.4}
  % Define x-coordinates
  \pgfmathsetmacro{\xstart}{0}
  \pgfmathsetmacro{\xstep}{0.5}

%\pi = 

  % Draw indices on left and right
  \foreach \index in {1, ..., \n} {
    \node at (\xstart - \indexOffset, \index) {$\scriptstyle \index$};
    \node at (\xstart + \xstep*2*\k + \indexOffset, \index) {$\scriptstyle \index$};
  }

% F_[1,2] F_[1,2]   
   \draw[-, ultra thick, dashed]
%\draw[-, very thick, color=red]
    (0,1) -- (0.5,1.5) -- (1,2) -- (1.5,1.5) -- (2,2);
%    \draw[-, very thick]
    \draw[-, very thick, color=blue]    
    (0,2) -- (0.5,1.5) -- (1,1) -- (1.5,1.5) -- (2,1);
     \fill (.5, 1.5) circle (1mm);  \fill (1.5, 1.5) circle (1mm);
    \fill (0, 1) circle (1mm);  \fill (0, 2) circle (1mm);
    \fill (2, 1) circle (1mm);  \fill (2, 2) circle (1mm);
     \node at (0.55, 1) {$\scriptstyle x_1$};  \node at (1.55, 1) {$\scriptstyle x_2$};
\end{tikzpicture}, 
\qquad \sigma = 
\begin{tikzpicture}[scale=.7,baseline=25]
  % Define variables
  \pgfmathsetmacro{\k}{2}  % Set k = 4
  \pgfmathsetmacro{\n}{2}  % Set n = 5
  \pgfmathsetmacro{\indexOffset}{0.4}
  % Define x-coordinates
  \pgfmathsetmacro{\xstart}{0}
  \pgfmathsetmacro{\xstep}{0.5}

  % Draw indices on left and right
  \foreach \index in {1, ..., \n} {
    \node at (\xstart - \indexOffset, \index) {$\scriptstyle \index$};
    \node at (\xstart + \xstep*2*\k + \indexOffset, \index) {$\scriptstyle \index$};
  }

% F_[1,2] F_[1,2]   
   \draw[-, ultra thick, dashed]
%   \draw[-, very thick, color=red]
    (0,1) -- (0.5,1.5) -- (1,1) -- (1.5,1.5) -- (2,2);
%    \draw[-, very thick]
    \draw[-, very thick, color=blue]    
    (0,2) -- (0.5,1.5) -- (1,2) -- (1.5,1.5) -- (2,1);
     \fill (.5, 1.5) circle (1mm);  \fill (1.5, 1.5) circle (1mm);
    \fill (0, 1) circle (1mm);  \fill (0, 2) circle (1mm);
    \fill (2, 1) circle (1mm);  \fill (2, 2) circle (1mm);
     \node at (0.55, 1) {$\scriptstyle x_1$};  \node at (1.55, 1) {$\scriptstyle x_2$};
\end{tikzpicture},
\end{equation}
%we modify a path family having one defect by 
%observe such a result from 
with $\dfct(\pi) = 1$, 
$\sigma$ constructed from $\pi$ by swapping the two $x_1$-to-$x_2$ subpaths of $\pi$,
and $\dfct(\sigma) = 0$.
%changing the crossing at vertex $x_1$ 
%the first internal vertex 
%to a non-crossing.
%, thus producing a path family with no defects,
On the other hand,
%However, 
this simple procedure does not always reduce defects by $1$.
%will not always produce the desired result of reducing the number of defects by $1$. 
Consider the star network and path families of type $1342$
\begin{equation}\label{eq:pi sigma}
F_{[2,4]} \bullet F_{[1,2]} \bullet F_{[2,4]} \bullet F_{[1,2]} = 
\begin{tikzpicture}[scale=.5,baseline=30]
  % Define variables
  \pgfmathsetmacro{\k}{4}  % Set k = 4
  \pgfmathsetmacro{\n}{4}  % Set n = 4
  \pgfmathsetmacro{\indexOffset}{0.4}
  % Define x-coordinates
  \pgfmathsetmacro{\xstart}{0}
  \pgfmathsetmacro{\xstep}{0.5}

  % Draw indices on left and right
  \foreach \index in {1, ..., \n} {
    \node at (\xstart - \indexOffset, \index) {$\scriptstyle \index$};
    \node at (\xstart + \xstep*2*\k + \indexOffset, \index) {$\scriptstyle \index$};
  }
  \node at (1.5, 2.6) {$\scriptstyle _{(2)}$};

   \draw[-, thick]
    (0,4) -- (0.5,3) -- (1,2) -- (2,1) -- (3,1) -- (3.5,1.5)  -- (4,2);

    \draw[-, thick]
    (0,3) -- (0.5,3) -- (2.5,3) -- (3,4) -- (4,4);

    \draw[-, thick]
    (0,2) -- (0.5,3) -- (4,3);

    \draw[-, thick]
    (0,1) -- (1,1) -- (2.5,3) -- (4,1);

    % \draw[-, thick]
    % (0,5) -- (0.5,5) -- (1,5) -- (1.5,4.5) -- (2,5) -- (2.5,5) -- (3,5) -- (3.5,4.5) -- (4,5);
   % \draw[-, thick]
   %  (0,1) -- (0.5,2.5) -- (1,1) -- (1.5,1) -- (2,1) -- (2.5,2.5) -- (3,1) -- (3.5,1) -- (4,1);

   %  \draw[-, thick]
   %  (0,2) -- (0.5,2.5) -- (1,2) -- (1.5,2) -- (2,2) -- (2.5,2.5) -- (3,2) -- (3.5,2) -- (4,2);

   %  \draw[-, thick]
   %  (0,3) -- (0.5,2.5) -- (1,3) -- (1.5,3) -- (2,3) -- (2.5,2.5) -- (3,3) -- (3.5,3) -- (4,3);

   %  \draw[-, thick]
   %  (0,4) -- (0.5,2.5) -- (1,4) -- (1.5,4.5) -- (2,4) -- (2.5,2.5) -- (3,4) -- (3.5,4.5) -- (4,4);

   %  \draw[-, thick]
   %  (0,5) -- (0.5,5) -- (1,5) -- (1.5,4.5) -- (2,5) -- (2.5,5) -- (3,5) -- (3.5,4.5) -- (4,5);

\fill (0, 1) circle (1.3mm);  \fill (0, 2) circle (1.3mm); \fill (0, 3) circle (1.3mm);  \fill (0, 4) circle (1.3mm);
\fill (4, 1) circle (1.3mm);  \fill (4, 2) circle (1.3mm); \fill (4, 3) circle (1.3mm);  \fill (4, 4) circle (1.3mm);
\fill (0.5, 3) circle (1.3mm);  \fill (1.43, 1.6) circle (1.3mm); \fill (2.5, 3) circle (1.3mm);  \fill (3.57, 1.6) circle (1.3mm);
   
\node at (0.75, 3.5) {$\scriptstyle x_1$};  
\node at (1.5, 0.85) {$\scriptstyle x_2$}; 
\node at (2.25, 3.5) {$\scriptstyle x_3$};
\node at (3.45, 0.85) {$\scriptstyle x_4$};
\end{tikzpicture}, 
\quad\pi = 
\begin{tikzpicture}[scale=.5,baseline=30]
  % Define variables
  \pgfmathsetmacro{\k}{4}  % Set k = 4
  \pgfmathsetmacro{\n}{4}  % Set n = 5
  \pgfmathsetmacro{\indexOffset}{0.4}
  % Define x-coordinates
  \pgfmathsetmacro{\xstart}{0}
  \pgfmathsetmacro{\xstep}{0.5}

  % Draw indices on left and right
  \foreach \index in {1, ..., \n} {
    \node at (\xstart - \indexOffset, \index) {$\scriptstyle \index$};
    \node at (\xstart + \xstep*2*\k + \indexOffset, \index) {$\scriptstyle \index$};
  }
%  Colors:
% https://www.overleaf.com/learn/latex/Using_colors_in_LaTeX
   \node at (1.5, 2.6) {$\scriptstyle{_{(2)}}$};

   \draw[-, ultra thick, dashed]
%   \draw[-, ultra thick, color=red]   
    (0,4) -- (0.5,2.9) -- (1,2) -- (2,1) -- (3,1) -- (3.5,1.5)  -- (4,2);

    \draw[-, ultra thick, color=green]
    (0,3) -- (0.5,3.1) -- (2.5,3.1) -- (3,4) -- (4,4);

    \draw[-, thick, dotted]
%    \draw[-, ultra thick, color=orange]    
    (0,2) -- (0.5,3) -- (4,3);

    \draw[-, ultra thick, color=blue]
    (0,1) -- (1,1) -- (2.5,2.9) -- (4,1);

\fill (0, 1) circle (1.3mm);  \fill (0, 2) circle (1.3mm); \fill (0, 3) circle (1.3mm);  \fill (0, 4) circle (1.3mm);
\fill (4, 1) circle (1.3mm);  \fill (4, 2) circle (1.3mm); \fill (4, 3) circle (1.3mm);  \fill (4, 4) circle (1.3mm);
\fill (0.5, 3) circle (1.3mm);  \fill (1.43, 1.6) circle (1.3mm); \fill (2.5, 3) circle (1.3mm);  \fill (3.57, 1.6) circle (1.3mm);
   
\node at (0.75, 3.5) {$\scriptstyle x_1$};  
\node at (1.5, 0.85) {$\scriptstyle x_2$}; 
\node at (2.25, 3.5) {$\scriptstyle x_3$};
\node at (3.45, 0.85) {$\scriptstyle x_4$};
   
\end{tikzpicture}, \quad
\sigma = 
\begin{tikzpicture}[scale=.5,baseline=30]
  % Define variables
  \pgfmathsetmacro{\k}{4}  % Set k = 4
  \pgfmathsetmacro{\n}{4}  % Set n = 5
  \pgfmathsetmacro{\indexOffset}{0.4}
  % Define x-coordinates
  \pgfmathsetmacro{\xstart}{0}
  \pgfmathsetmacro{\xstep}{0.5}

  % Draw indices on left and right
  \foreach \index in {1, ..., \n} {
    \node at (\xstart - \indexOffset, \index) {$\scriptstyle \index$};
    \node at (\xstart + \xstep*2*\k + \indexOffset, \index) {$\scriptstyle \index$};
  }
%  Colors:
% https://www.overleaf.com/learn/latex/Using_colors_in_LaTeX
    \node at (1.5, 2.6) {$\scriptstyle{_{(2)}}$};

    \draw[-, ultra thick, dashed]
%   \draw[-, ultra thick, color=red]   
    (0,4) -- (0.5,2.9) -- (1,2) -- (1.5, 1.5) -- (2.5, 2.9) -- (3.5, 1.5) -- (4,2);

    \draw[-, ultra thick, color=green]
    (0,3) -- (0.5,3.1) -- (2.5,3.1) -- (3,4) -- (4,4);

    \draw[-, thick, dotted]
%    \draw[-, ultra thick, color=orange]    
    (0,2) -- (0.5,3) -- (4,3);

    \draw[-, ultra thick, color=blue]
    (0,1) -- (1,1) -- (1.5, 1.5) -- (2, 1) -- (3, 1) -- (3.5, 1.5) -- (4,1);

\fill (0, 1) circle (1.3mm);  \fill (0, 2) circle (1.3mm); \fill (0, 3) circle (1.3mm);  \fill (0, 4) circle (1.3mm);
\fill (4, 1) circle (1.3mm);  \fill (4, 2) circle (1.3mm); \fill (4, 3) circle (1.3mm);  \fill (4, 4) circle (1.3mm);
\fill (0.5, 3) circle (1.3mm);  \fill (1.43, 1.6) circle (1.3mm); \fill (2.5, 3) circle (1.3mm);  \fill (3.57, 1.6) circle (1.3mm);
   
\node at (0.75, 3.5) {$\scriptstyle x_1$};  
\node at (1.5, 0.85) {$\scriptstyle x_2$}; 
\node at (2.25, 3.5) {$\scriptstyle x_3$};
\node at (3.45, 0.85) {$\scriptstyle x_4$};

\end{tikzpicture},
\end{equation}
with $\dfct(\pi) = 1$, and $\sigma$ constructed from $\pi$ by
swapping the $x_2$-to-$x_4$ subpaths of
%changing the red and blue paths, 
$\pi_1$ and $\pi_4$.  The swap
%respectively, from crossing to not crossing in $x_k$, 
eliminates the defect $(\pi_1, \pi_4, 4)$,
but introduces two more:
%while introducing 2 new defects: 
\begin{equation}\label{eq:2 sigma defects}
    (\sigma_2, \sigma_4, 3), \quad (\sigma_3, \sigma_4, 3).
\end{equation}
Thus we have $\dfct(\sigma) = 2$.
There is in fact a path family 
\begin{equation}\label{eq:tau}
\tau = 
\begin{tikzpicture}[scale=.5,baseline=30]
  % Define variables
  \pgfmathsetmacro{\k}{4}  % Set k = 4
  \pgfmathsetmacro{\n}{4}  % Set n = 5
  \pgfmathsetmacro{\indexOffset}{0.4}
  % Define x-coordinates
  \pgfmathsetmacro{\xstart}{0}
  \pgfmathsetmacro{\xstep}{0.5}

  % Draw indices on left and right
  \foreach \index in {1, ..., \n} {
    \node at (\xstart - \indexOffset, \index) {$\scriptstyle \index$};
    \node at (\xstart + \xstep*2*\k + \indexOffset, \index) {$\scriptstyle \index$};
  }
%  Colors:
% https://www.overleaf.com/learn/latex/Using_colors_in_LaTeX
   \node at (1.5, 2.4) {$\scriptstyle _{(2)}$};

   \draw[-, ultra thick, dashed]
%\draw[-, ultra thick, color=red]
    (0,4) -- (0.5, 3.1) -- (2.5, 3.1)  -- (3.5, 1.5) -- (4,2);

    \draw[-, ultra thick, color=green]
    (0,3) -- (2.5,3) -- (3,4) -- (4,4);

    \draw[-, thick, dotted]
%\draw[-, ultra thick, color=orange]
    (0,2) -- (0.5, 2.9) -- (1, 2) -- (1.5, 1.5) -- (2.5, 3)  -- (4,3);

    \draw[-, ultra thick, color=blue]
    (0,1) -- (1,1) -- (1.5, 1.5) -- (2, 1) -- (3, 1) -- (3.5, 1.5) -- (4,1);

\fill (0, 1) circle (1.3mm);  \fill (0, 2) circle (1.3mm); \fill (0, 3) circle (1.3mm);  \fill (0, 4) circle (1.3mm);
\fill (4, 1) circle (1.3mm);  \fill (4, 2) circle (1.3mm); \fill (4, 3) circle (1.3mm);  \fill (4, 4) circle (1.3mm);
\fill (0.5, 3) circle (1.3mm);  \fill (1.43, 1.6) circle (1.3mm); \fill (2.5, 3) circle (1.3mm);  \fill (3.57, 1.6) circle (1.3mm);
   
\node at (0.75, 3.5) {$\scriptstyle x_1$};  
\node at (1.5, 0.85) {$\scriptstyle x_2$}; 
\node at (2.25, 3.5) {$\scriptstyle x_3$};
\node at (3.45, 0.85) {$\scriptstyle x_4$};
   
\end{tikzpicture}
\end{equation}
of type $1342$ satisfying 
%$\type(\tau) = \type(\pi)$ and 
$\dfct(\tau) = \dfct(\pi)-1 = 0$, but the naive approach above does not produce it from $\pi$.
%of the same type that has $0$ defects.

%To describe the process of "removing" one defect from a path family, we begin
To describe the process of reducing defects in a path family by exactly $1$, we begin
by stating a map which modifies a path family by removing a defect at a specified vertex,
%$x_k$ for $k \geq 2$,
possibly creating earlier defects.
%Since a defect does not take place in the initial star of $F$, we only consider $k\geq 2$. 
For 
\begin{equation}\label{eq:starnetwork}
    F = F_{[a_1,b_1]} \cdots F_{[a_m,b_m]} \in \ocnet n \quad \text{having internal vertices } \ctr 1, \dotsc, \ctr m
\end{equation}
and $k\in\{2, \dotsc, m\}$, 
%we define a function 
define
\begin{equation}\label{eq:phidef}
\begin{aligned}
    \phi_k : \{\pi \in \Pi(F) \,|\, \pi \text{ has a defect at } x_k
    %of the form }(\pi_i, \pi_j, k) \text{ for some }  i, j
    \} &\rightarrow \Pi(F)\\
    \pi &\mapsto \hat \pi\end{aligned}
    \end{equation}
by the following algorithm.
\begin{alg}\label{a:remove_early_defect_random_k}
    Given star network 
    $F$ as in \eqref{eq:starnetwork},
    % = F_{[a_1, b_1]} \bullet %F_{[a_2, b_2]} \bullet 
    % \cdots \bullet F_{[a_m, b_m]} \in \cnet n$
    %$w \in \sn$, $d \geq 1$ 
    path family $\pi \in \Pi(F)$,
    and index $k$ such that $\pi$ has a defect at $\ctr k$,
    %for some $w \in \sn$, $d \geq 1$, and $k \in \{2, \ldots, m\}$ where there exists some defect of the form $(\pi_i, \pi_j, k)$,
    do 
    \begin{enumerate}
    \item Let ($r, t$) be the lexicographically least pair such that $(\pi_r, \pi_t, k)$ is a defect.
    \item %Fix $\pi_t$ and 
    %Choose path $\pi_s$ where 
    Let $s$ be the largest index such that $\pi_s$ enters vertex $\ctr k$
    %star $F_{[a_k, b_k]}$ 
    through the same edge as $\pi_r$ and $(\pi_s, \pi_t, k)$ is a defect.
    \item Let $\ctr \ell$ be the final vertex in the %earliest 
    rightmost crossing of 
    %$\pi_s \cap \pi_t$ where 
    $\pi_s$ and $\pi_t$ prior to $\ctr k$. 
    %$F_l = F_{[a_l, b_l]}$ be the star where $\pi_s$ and $\pi_t$ 
    %first cross.
    %(what does this mean for condensed concatenations? this would be the same with regular concatenations, right? Perhaps we need to clarify that paths are in parallel when they share an edge.), which must occur earlier than the defect at $F_k = F_{[a_k, b_k]}$. 
    \item Create a new path family $%\phi_k(\pi) =
    \hat \pi = (\hat \pi_1, \dotsc, \hat \pi_n)$ by \begin{enumerate}
        \item $\hat \pi_i = \pi_i$ if $i \not \in \{s,t\}$.
        \item $\hat \pi_s$ is $\pi_s$ with the
        $\ctr \ell$-to-$\ctr k$ subpath replaced by that of $\pi_t$.
        %, $\pi_t$ from $\ctr l$ to $\ctr k$, and $\pi_s$ from $\ctr k$ to sink $s$.
        %is $\pi_t$ between central vertices of stars $F_l$ and $F_k$ and is $\pi_s$ otherwise.
        \item $\hat \pi_t$ is $\pi_t$ with the
        $\ctr \ell$-to-$\ctr k$ subpath replaced by that of $\pi_s$.
        %$\pi_s$ between central vertices of stars $F_l$ and $F_k$ and is $\pi_t$ otherwise.
       \end{enumerate}
    \end{enumerate}
\end{alg}

For example, applying $\phi_3$ to the path family $\sigma$ in \eqref{eq:pi sigma}, we consider the defects \eqref{eq:2 sigma defects} and set $(r, t) = (2, 4)$. Now we consider the paths entering $\ctr 3$ through the same edge as $\sigma_2$ and also creating a defect at $\ctr 3$ with $\sigma_4$. Of these, $\sigma_3$ has the greatest index and we set $s = 3$. Prior to $\ctr 3$, the rightmost crossing of $\sigma_3$ and $\sigma_4$ is the vertex $\ctr 1$, so we set $\ell = 1$. We create a new path family 
$\hat \sigma  = (\hat \sigma_1, \hat \sigma_2, \hat \sigma_3, \hat \sigma_4)$ by swapping the $\ctr 1$-to-$\ctr 3$ subpaths of $\sigma_3$ and $\sigma_4$,

% we obtain $\sigma$ in \eqref{eq:pi sigma}.

\begin{equation*}
\hat \sigma = \phi_3(\sigma) = 
\begin{tikzpicture}[scale=.5,baseline=30]
  % Define variables
  \pgfmathsetmacro{\k}{4}  % Set k = 4
  \pgfmathsetmacro{\n}{4}  % Set n = 5
  \pgfmathsetmacro{\indexOffset}{0.4}
  % Define x-coordinates
  \pgfmathsetmacro{\xstart}{0}
  \pgfmathsetmacro{\xstep}{0.5}

  % Draw indices on left and right
  \foreach \index in {1, ..., \n} {
    \node at (\xstart - \indexOffset, \index) {$\scriptstyle \index$};
    \node at (\xstart + \xstep*2*\k + \indexOffset, \index) {$\scriptstyle \index$};
  }
%  Colors:
% https://www.overleaf.com/learn/latex/Using_colors_in_LaTeX
    \node at (1.5, 2.6) {$\scriptstyle{_{(2)}}$};

    \draw[-, ultra thick, dashed]
%   \draw[-, ultra thick, color=red]   
    (0,4) -- (0.5,3.15) -- (2.5, 3.15) -- (3.5, 1.5) -- (4,2);

    \draw[-, ultra thick, color=green]
    (0,3) -- (0.5,2.9) -- (1,2) -- (1.5, 1.5) -- (2.5,2.9) -- (3,4) -- (4,4);

    \draw[-, thick, dotted]
%    \draw[-, ultra thick, color=orange]    
    (0,2) -- (0.5,3) -- (4,3);

    \draw[-, ultra thick, color=blue]
    (0,1) -- (1,1) -- (1.5, 1.5) -- (2, 1) -- (3, 1) -- (3.5, 1.5) -- (4,1);

\fill (0, 1) circle (1.3mm);  \fill (0, 2) circle (1.3mm); \fill (0, 3) circle (1.3mm);  \fill (0, 4) circle (1.3mm);
\fill (4, 1) circle (1.3mm);  \fill (4, 2) circle (1.3mm); \fill (4, 3) circle (1.3mm);  \fill (4, 4) circle (1.3mm);
\fill (0.5, 3) circle (1.3mm);  \fill (1.43, 1.6) circle (1.3mm); \fill (2.5, 3) circle (1.3mm);  \fill (3.57, 1.6) circle (1.3mm);
   
\node at (0.75, 3.5) {$\scriptstyle x_1$};  
\node at (1.5, 0.85) {$\scriptstyle x_2$}; 
\node at (2.25, 3.5) {$\scriptstyle x_3$};
\node at (3.45, 0.85) {$\scriptstyle x_4$};

\end{tikzpicture}.
\end{equation*}
We leave as an exercise for the reader the verification that the path families $\pi, \sigma, \hat \sigma, \tau$ are related by $\sigma = \phi_4(\pi)$, $\tau = \phi_3(\hat \sigma)$.

% ($\star$ Can we call the algorithm(s) some function $\phi_k$?)
It is straightforward to verify that the path families $\pi$, $\hat \pi$ in Algorithm~\ref{a:remove_early_defect_random_k} have the same type, and that their defects at $x_{k+1},\dotsc,x_m$ correspond bijectively.
%The path family $\hat \pi$ produced by Algorithm~\ref{a:remove_early_defect} produces a path family $\hat \pi$ which has the 
\begin{prop}\label{p: one less defect condensed-adjusted}
Given star network 
    $F$ as in \eqref{eq:starnetwork}, 
    %$w \in \sn$, $d \geq 1$ 
    path family $\pi \in \Pi(F)$
    and index $k$ such that $\pi$ has a defect at $\ctr k$,
    % Fix $F \in \cnet n$, $w \in \sn$, $\pi \in \Pi_w(F)$, and index $k$ such that $\pi$ has a defect of the form $(\pi_i, \pi_j, k)$.
    the path family $\hat\pi = \phi_k(\pi)$ satisfies
    \begin{enumerate}
        \item $\type(\hat \pi) = \type(\pi)$,
        \item for each $p > k$, we have $\{(i,j) \,|\, (\hat\pi_i, \hat\pi_j, p) \text{ defective}\} = \{(i,j) \,|\, (\pi_i, \pi_j, p) \text{ defective}\}$.
\end{enumerate}
\end{prop}

%($\star$ Finished proof Monday October 27. Reorganized proof Monday November 17.)  

% ($\star$ Make a decision on Monday November 3 where to put the definition of this preorder). Would it help to define a preorder $\precsim$ on the paths entering some vertex $x$ by $\pi_i \precsim \pi_j$ if $\pi_i$ enters vertex $x$ at or below where $\pi_j$ enters $x$, and $\pi_i \prec \pi_j$ if the entry edges for the two paths are distinct? Note that this isn't a partial order: $\pi_i \precsim \pi_j$ and $\pi_j \precsim \pi_i$ do not imply that $\pi_i = \pi_j$; they imply that $\pi_i \sim \pi_j$, i.e., that $\pi_i$ and $\pi_j$ enter $x_k$ on the same edge.)

\begin{proof}
    %\noindent 
    
    %It is clear that we have $\hat\pi \in \Pi_w(F)$, since $\pi \in \Pi_w(F)$, and 
     %Since 
    (1) Since $\hat \pi$ differs from $\pi$ only in subpaths between interior vertices $\ctr \ell$ and $x_k$, we have that for $i =1, \dotsc, n$,
    path $\hat \pi_i$ has the same source and sink vertices as $\pi_i$.  
    %we have
    %$\type(\hat \pi) = \type(\pi)$.
    %Since $\hat \pi$ covers $F$ we have $\hat \pi \in \Pi_w(F)$.  ($\star$ necessary?)
%    , $\hat\pi_i = \pi_i$ for $i \notin \{s, t\}$, and $\hat\pi_i, \pi_i$ both terminate at sink $w_i$ of $F$ for $i\in\{s, t\}$.

(2) %Suppose first that $\{i,j\} = \{s,t\}$.
%The component of $\pi_s \cap \pi_t$ containing $\ctr l$ is a crossing, while the component of $\hat\pi_s \cap \hat\pi_t$ containing $\ctr l$ is a noncrossing;
%the component of $\pi_s \cap \pi_t$ containing $\ctr k$ is a crossing if and only if the component
%of $\hat\pi_s \cap \hat\pi_t$ containing $\ctr k$ is a noncrossing.
%Thus each later component of $\pi_s \cap \pi_t$
%is preceded by an odd number of crossings if and only if
%the corresponding component of $\hat \pi_s \cap \hat \pi_t$ is as well.
%Now suppose that $\{i,j \}\cap  \{s,t\} = \emptyset$.  Since $\hat \pi_i = \pi_i$, $\hat \pi_j = \pi_j$, every triple $(\pi_i, \pi_j, p)$ is defective if and only if $(\hat \pi_i, \hat \pi_j, p)$ is.  Now suppose that $|\{i,j\} \cap \{s,t\}| = 1$.  ($\star$ Explain why the relevant pairs triples are both defective or both nondefective.)
Since the restrictions of $\hat\pi$ 
%is identical to 
and $\pi$ to $F_{[a_{k+1}, b_{k+1}]} \cdots F_{[a_m, b_m]}$ are identical, we have for each triple $(i,j,p)$ with $i < j$ and $p > k$ that $\pi_i$ enters $\ctr p$ above $\pi_j$ if and only if $\hat \pi_i$ enters $\ctr p$ above $\hat \pi_j$.
By Lemma~\ref{l: alt defect defn}, $(\pi_i,\pi_j,p)$ is defective if and only if $(\hat\pi_i, \hat\pi_j, p)$ is.
\end{proof}
%and the complementary truncations satisfy 
%we have 
%$\type(\trunc k(\pi)) = \type(\trunc k(\hat\pi))$, Lemma~\ref{l: dfct history} or Lemma~\ref{l:laterdefects} implies that
%, then 
%all defects of $\hat\pi$ on the latter portion of $F$ correspond bijectively with defects of $\pi$ there.
%($\star$ Restate Lemma~\ref{l:laterdefects} and/or Lemma~\ref{l: dfct history} and apply here?)
%\textcolor{blue}{Since $\pi$ and $\hat \pi$ differ by one crossing of paths indexed by $s$, $t$ on $F_1 \bullet \cdots \bullet F_k$, doesn't each defective meeting of $\pi_s$ and $\pi_t$ after $F_k$ become a nondefective meeting of $\hat \pi_s$ and $\hat\pi_t$, and vice versa?}

    %\noindent 
    On the other hand, path families $\pi$, $\hat \pi$ 
    %\eqref{eq: phi_k}  
    need not have a bijective correspondence of defects occuring before $x_k$.
    For example in \eqref{eq:pi sigma} -- \eqref{eq:2 sigma defects}, the path family $\pi$ has 0 defects at $\ctr 3$ while $\sigma = \phi_4(\pi)$ has 2.
    
    We will show in Proposition~\ref{p:remove one defect} that
    $\hat \pi$ has exactly one fewer defect at $x_k$ than $\pi$ does.  The ``extra" defect of $\pi$ is easy to describe.    
%($\star$ Need to polish the following to introduce partitions of defects, and define $\psi$.) 
  Let 
  %$r$, 
  $s$, $t$
%, $l$ 
be the numbers computed in steps 1 -- 2 of 
Algorithm~$\ref{a:remove_early_defect_random_k}$ and order the paths entering $\ctr k$ by $\prec$ as in Section 3.
%the algorithm.
Then 
%$r$, 
$s$, $t$ satisfy 
\begin{equation}\label{eq:st}
%r \leq 
s < t, 
\qquad
\pi_t \prec 
%\pi_r, 
\pi_s,
\qquad
%and
%with the two paths meeting defectively;
%$\pi_s$ and $\pi_t$ meet defectively at $\ctr k$, with 
\hat\pi_s \prec \hat\pi_t,
\end{equation}
and by Lemma~\ref{l: alt defect defn}, the triple
$(\pi_s, \pi_t, k)$ is defective and the triple $(\hat\pi_s,\hat\pi_t,k)$ is not.
We claim that other defects of $\pi$ at $\ctr k$ correspond bijectively to all defects of $\hat \pi$ at $\ctr k$.
First, it is easy to see that for $i, j \not\in \{ s, t\}$, the triple
$(\pi_i,\pi_j,k)$ is defective if and only if the triple $(\hat \pi_i, \hat\pi_j, k)$ is:
in this case we have $\hat \pi_i = \pi_i$ and $\hat \pi_j = \pi_j$.
Next, we look closely
at defects $(\pi_i, \pi_j, k)$ and $(\hat \pi_i, \hat \pi_j, k)$ 
%We claim that defects of $\pi$ at $\ctr k$ different from $(\pi_s, \pi_t, k)$
%correspond bijectively to defects of $\hat\pi$ at $\ctr k$. 
%
%We therefore consider such defects %$(\pi_i,\pi_j,k)$, $(\hat\pi_i,\hat\pi_j,k)$ 
with exactly one of the indices $\{i, j\}$ belonging to $\{s,t\}$,
\begin{equation}\label{eq:ijst3}
    |\{i,j\}\cap\{s,t\}|=1,
\end{equation}
and show that these correspond bijectively. Call these sets of defects $\D(\pi)$ and $\D(\hat\pi)$, respectively,
and define four subsets of each:
%as follows.
\begin{equation}\label{eq:ABCC}
    \begin{aligned}
        \A &\defeq \{ (\pi_i,\pi_j,k) \in \D(\pi) \,|\, \pi_j \prec \pi_t \}, \qquad
        \A' \defeq \{ (\hat \pi_i,\hat \pi_j,k) \in \D(\hat \pi) \,|\, \hat \pi_j \prec \hat \pi_s \},\\
        \B &\defeq \{ (\pi_i,\pi_j,k) \in \D(\pi) \,|\, \pi_s \prec \pi_i \}, \qquad
        \B' \defeq \{ (\hat\pi_i,\hat\pi_j,k) \in \D(\hat\pi) \,|\, \hat\pi_t \prec \hat\pi_i \},\\
        \C_1 &\defeq \{ (\pi_i,\pi_j,k) \in \D(\pi) \,|\, \pi_t\precsim\pi_j\prec \pi_i=\pi_s \text{ and } \pi_t\neq \pi_j \},\\
        &\qquad\qquad\qquad 
        \C'_1 \defeq \{ (\hat\pi_i,\hat\pi_j,k) \in \D(\hat\pi) \,|\, \hat\pi_s\precsim\hat\pi_j\prec \hat\pi_i=\hat\pi_t \text{ and } \hat\pi_s\neq \hat\pi_j \},\\
        \C_2 &\defeq \{ (\pi_i,\pi_j,k) \in \D(\pi) \,|\, \pi_t=\pi_j\prec \pi_i\precsim \pi_s \text{ and } \pi_i\neq \pi_s \},\\
        &\qquad\qquad\qquad 
        \C'_2 \defeq \{ (\hat\pi_i,\hat\pi_j,k) \in \D(\hat\pi) \,|\, \hat\pi_s=\hat\pi_j\prec \hat\pi_i\precsim \hat\pi_t \text{ and } \hat\pi_i\neq \hat\pi_t \}.
%        \pi_t\precsim\pi_j\prec \pi_i=\pi_s \text{ and } \pi_t\neq \pi_j \},\\
    \end{aligned}
\end{equation}
\begin{lem}\label{l:ABCD}
    The sets \eqref{eq:ABCC} partition $\D(\pi)$ and $\D(\hat\pi)$ into disjoint blocks:
    %.  In particular we have
    \begin{equation}\label{eq:blocks}
        \D(\pi) = \A \uplus \B \uplus \C_1 \uplus \C_2,
        \qquad
        \D(\hat\pi) = \A' \uplus \B' \uplus \C'_1 \uplus \C'_2.
    \end{equation}
\end{lem}
\begin{proof}
    % ($\star$ Briefly explain this, probably using some of the following and a few more words.  In particular, explain where the two subcases of $c$ come from.)
    %\end{proof}
%By \eqref{eq:st}, 

To see the first equality of \eqref{eq:blocks}, observe that by definition of $\D(\pi)$, we have 
\begin{equation}\label{eq:ij}
    i<j, \qquad \pi_j\prec \pi_i. 
\end{equation}
% We claim that we cannot have $\pi_j\prec \pi_t$ and $\pi_s\prec \pi_i$ simultaneously: by \eqref{eq:st} and \eqref{eq:ij}, this implies $\pi_j\prec \pi_t\prec\pi_s\prec \pi_i$, contradicting \eqref{eq:ijst3}. Thus we have $\A\cap \B=\emptyset$.
% Combining \eqref{eq:st}, \eqref{eq:ijst3} and \eqref{eq:ij}, we observe we cannot have $\pi_j\prec \pi_t$ and $\pi_s\prec \pi_i$ simultaneously. This would imply that we have $\pi_j\neq \pi_t$ and $\pi_s\neq \pi_i$ simultaneously. 
Consider a defect $(\pi_i, \pi_j, k) \in \D(\pi)$. If $\pi_j\prec \pi_t$, the defect belongs to $\A$. If $\pi_s\prec\pi_i$, it belongs to $\B$. On the other hand, we cannot have $\pi_j\prec \pi_t$ and $\pi_s\prec \pi_i$ simultaneously: by \eqref{eq:st} and \eqref{eq:ij}, this implies $\pi_j\prec \pi_t\prec\pi_s\prec \pi_i$, contradicting \eqref{eq:ijst3}. Thus we have $\A\cap \B=\emptyset$. Now suppose we have $\pi_j\not\prec\pi_t$ and $\pi_s\not\prec\pi_i$, i.e., $\pi_t\precsim \pi_j\prec\pi_i\precsim \pi_s$. By \eqref{eq:ijst3}, this implies that either $\pi_t\precsim\pi_j\prec \pi_i=\pi_s \text{ and } \pi_t\neq \pi_j$, or $\pi_t=\pi_j\prec \pi_i\precsim \pi_s \text{ and } \pi_i\neq \pi_s$, and the defect belongs to $\C_1$ or $\C_2$, respectively.

The second equality of \eqref{eq:blocks} is similar.
\end{proof}

To demonstrate the claimed bijection between $\D(\pi)$ and $\D(\hat\pi)$, we define the map
\begin{equation}\label{eq:psidef}
\begin{aligned}
\psi: \D(\pi)&\to \D(\hat\pi),\\
(\pi_i,\pi_j,k)&\mapsto 
\begin{cases} 
(\hat\pi_i, \hat\pi_j,k),\text{ if } (\pi_i,\pi_j,k) \in \A \cup \B,\\
%Case (1) or (2), } \text{ or if } \pi_i, \pi_j, \pi_s, \pi_t \text{ satisfy } (a) \text{ or } (b),\\ 
(\hat\pi_t, \hat\pi_j,k), \text{ if } (\pi_i, \pi_j, k) = (\pi_s, \pi_j, k) \in \C_1,\\
%\pi_s, \pi_t \text{ satisfy } (c$\,-\,$i) ,\\ 
(\hat\pi_i,\hat\pi_s,k), \text{ if } (\pi_i, \pi_j, k) = (\pi_i, \pi_t, k) \in \C_2,
%\pi_s, \pi_t \text{ satisfy } (c$\,-\,$ii)
\end{cases}
\end{aligned}
\end{equation}
and prove that it is bijective.
%is a bijection.

\begin{lem}\label{l: well defined}
The map $\psi$ is well defined: for $(\pi_i, \pi_j, k)$ defective, we have that
$\psi(\pi_i, \pi_j, k)$ is also defective.  Furthermore, we have
%\begin{equation*}
    $\psi(\A) \subseteq \A'$,
%    \qquad
    $\psi(\B) \subseteq \B'$,
%    \qquad
    $\psi(\C_1) \subseteq \C'_1$,
%    \qquad
    $\psi(\C_2) \subseteq \C'_2$.
%\end{equation*}
\end{lem}

\begin{proof}
%To see that $\phi(\pi_i,\pi_j,k)$ is defective 
%%belongs to $\D(\hat\pi)$ 
%when $(\pi_i,\pi_j,k)$ is,
%Consider the 
%four 
%blocks 
%$\A$ \eqref{eq:ABCC} -- \eqref{eq:blocks} 
%of $\D(\pi)$.
For $(\pi_i,\pi_j,k) \in \A$, the conditions \eqref{eq:st} imply that we have
% \begin{enumerate}[(a)]
    % \item 
%    $(a)$ The conditions $(\pi_i,\pi_j,k)\in\D(\pi)$ and \eqref{eq:st} imply that 
$\pi_j\prec \pi_t\prec \pi_s$ and $i\in\{s,t\}$. 
By Algorithm~\ref{a:remove_early_defect_random_k}, the edges on which $\hat\pi_s$, $\hat\pi_t$ enter $x_k$ are those on which
$\pi_t$, $\pi_s$ enter $x_k$, respectively. 
Thus we have $\hat\pi_j\prec \hat\pi_s\prec \hat\pi_t$, implying that $\hat\pi_j\prec \hat\pi_i$, and
%It follows that
$\psi(\pi_i,\pi_j,k) = (\hat\pi_i,\hat\pi_j,k)$ is defective.
%and
%It follows that 
%$(\hat\pi_i, \hat\pi_j,k) \in 
%\A' \subseteq 
%\D(\hat\pi)$. 
In particular, we have
$(\hat\pi_i, \hat\pi_j,k) \in 
\A'$.
%    $(b)$ Similar. 
Similarly, for $(\pi_i,\pi_j,k) \in \B$ we have that $\psi(\pi_i,\pi_j,k) = (\hat\pi_i,\hat\pi_j,k)$ 
is defective and 
belongs to $\B'$.

Now let $r$ be the number computed in step $1$ of Algorithm~\ref{a:remove_early_defect_random_k}, satisfying
\begin{equation}\label{eq:rst}
    r \leq s < t, \qquad \pi_t \prec \pi_r, \pi_s, \qquad \pi_r \sim \pi_s.
\end{equation}
%belongs to $\D(\pi)$,
%($\star$ Now 
%put $r$ into the $s$-$t$ equation: 
%$r \leq s < t$ and 
%\qquad
%$\pi_t \prec \pi_r, \pi_s$, and 
%clean up the other two cases.)

For $(\pi_i, \pi_j, k) \in \C_1$, 
%$(c\,$-$\,i)$ We 
the conditions 
%we have 
$\pi_t\precsim \pi_j\prec \pi_i=\pi_s$ and $\pi_t\neq \pi_j$ 
%guarantee that 
and Algorithm~\ref{a:remove_early_defect_random_k} 
%produces $\hat\pi$ with 
%we see that 
guarantee that
we have
$\hat\pi_s \precsim \hat\pi_j\prec \hat\pi_t$. 
%We claim that $t < j$ and
To see that $\psi(\pi_i,\pi_j,k) = (\hat\pi_t, \hat\pi_j,k)$ is defective and belongs to $\C'_1$, it remains to show that $t<j$ and 
%that 
$\hat \pi_s \neq \hat \pi_j$. 
Since we have $\pi_j \prec \pi_i=\pi_s\sim\pi_r$ with $r\leq s = i < j$, we know that $(\pi_r,\pi_j,k)$ is defective and that $\hat \pi_s \neq \hat \pi_j$. 
%($\star$ Is another fact missing here?  Are hats missing on the paths in some triples?) 
By our choice of $r$ in step $1$ of Algorithm~\ref{a:remove_early_defect_random_k}, we must have $(r, t)\leq_{\mathrm{lex}}(r,j)$ and therefore $t\leq j$. Since $\pi_t \neq \pi_j$, we 
%do indeed 
have $t<j$, as desired. 
%($\star$ what reason?) we conclude that $(\hat\pi_t, \hat\pi_j,k)$ belongs to $\C'_1$. ($\star$ by definition of $\C'_1$ it contains defects where $i=t$)

%    $(c\,$-$\,ii)$ 
    
 For $(\pi_i, \pi_j, k) \in \C_2$, the conditions
 %We have that 
 $\pi_t=\pi_j\prec \pi_i\precsim \pi_s$ and $\pi_i\neq \pi_s$ 
and 
%guarantee that 
Algorithm~\ref{a:remove_early_defect_random_k} 
%produces $\hat\pi$ with 
guarantee that
 we have $\hat\pi_s\prec \hat\pi_i \precsim \hat\pi_t$ and $\hat\pi_i \neq \hat\pi_t$. 
 To see that $\psi(\pi_i, \pi_j, k) = (\hat\pi_i, \hat\pi_s,k)$ is defective and belongs to $\C'_2$,
 %$\in\D(\hat\pi)$, 
it remains to show that $i<s$.
%and $\hat \pi_i \neq \hat\pi_t$. 
Since $\pi_i \precsim \pi_s$ and by \eqref{eq:rst}, $\pi_i, \pi_s, \pi_r$ are related by $\pi_i \prec \pi_s \sim \pi_r$ or $\pi_i \sim \pi_s \sim \pi_r$.
%We consider the two cases of $\pi_i \precsim \pi_s$.
%($\star$ Polish the rest.) 
%First consider the case that $\pi_i\prec \pi_s\sim \pi_r$.
Consider the first case. We have $i < j = t$ so $(r, i) <_{\mathrm{lex}} (r, t)$, and by step 1 of Algorithm~\ref{a:remove_early_defect_random_k} $(\pi_r, \pi_i, k)$ is not defective. We know that $\pi_i \prec \pi_r$ so $i < r$. By \eqref{eq:rst} we also have $r \leq s$, and thus $i < s$.
% Suppose that $s < i$ in the first case. By \eqref{eq:rst} we have $r < i$, so $(\pi_r, \pi_i, k)$ is defective. But $i < j = t$, so $(r, i) <_{\mathrm{lex}} (r, t)$, and by step 1 of Algorithm~\ref{a:remove_early_defect_random_k} $(\pi_r, \pi_i, k)$ cannot be defective. Thus $i < s$.
%Suppose $r\leq s<i$, we have that $(\pi_r, \pi_i,k)$ is defective. But $i<j$, then $r\leq s<i<j=t$, implying $(\pi_r, \pi_i,k)$ to be defective. But $(r,i)<(r,t)$, contradicting to the fact that $(r,t)$ is the lex-least pair. Therefore, we cannot have $i>s$ with $\pi_i\prec \pi_s$. 
Now consider the second case. Since $\pi_i \sim \pi_r$ and $(\pi_i, \pi_t,k)$ is defective, step 2 of Algorithm~\ref{a:remove_early_defect_random_k} gives $i < s$.
\end{proof}
% \end{enumerate}
%Since $(\pi_i,\pi_j,k)=(\pi_i, \pi_t,k)$ is defective. By Algorithm~\ref{a:remove_early_defect_random_k}, we have $i<s$.

\begin{lem}\label{lem: injective}
The map $\psi$ is injective. 
\end{lem}

\begin{proof}
%To see that $\psi$ is injective, let $(\pi_{i_1}, \pi_{j_1}, k), (\pi_{i_2}, \pi_{j_2}, k)$ be distinct defects in $\D(\pi)$, and consider all combinations of the two defects falling into cases $(a)$ -- $(c\,$-$\,ii)$.
% ($a$) 
By Lemmas~\ref{l:ABCD} -- \ref{l: well defined}, we consider pairs $(\pi_{i_1}, \pi_{j_1}, k) , (\pi_{i_2}, \pi_{j_2}, k)$ of distinct defects in $\A$, $\B$, $\C_1$, $\C_2$ separately.

%Consider distinct 
%For \in \A$, the map
%fall into case $(a)$. For $(\pi_{i_2}, \pi_{j_2}, k)$ in case $(a)$ or $(b)$, 
If the two defects belong to $\A \cup \B$, then $\psi$ clearly maps them 
%two defects 
to $(\hat\pi_{i_1}, \hat\pi_{j_1}, k) \neq (\hat\pi_{i_2}, \hat\pi_{j_2}, k)$. 
%If the second triple in $\pi$ falls in case $(c$\,-\,$i)$ instead, $\psi$ maps it to $(\hat\pi_t, \hat\pi_{j_2}, k)$, which differs from $(\hat\pi_{i_1}, \hat\pi_{j_1}, k)$ since $\pi_{j_1} \prec \pi_t \precsim \pi_{j_2}$ so $j_1 \neq j_2$. If the second triple in $\pi$ falls in case $(c$\,-\,$ii)$ instead, $\psi$ maps it to $(\hat\pi_{i_2}, \hat\pi_s, k)$, which differs from $(\hat\pi_{i_1}, \hat\pi_{j_1}, k)$ since $\pi_{j_1} \prec \pi_t \prec \pi_s$, so $j_1 \neq s$.
%Let $(\pi_{i_1}, \pi_{j_1}, k)$ fall into case $(b)$. Again, $\psi$ maps $(\pi_{i_1}, \pi_{j_1}, k)$ to $(\hat\pi_{i_1}, \hat\pi_{j_1}, k)$. Now consider the second triple $(\pi_{i_2}, \pi_{j_2}, k)$ falls in case $(b)$. Then $\psi$ maps it to $(\hat\pi_{i_2}, \hat\pi_{j_2}, k)$, which differs from $(\hat\pi_{i_1}, \hat\pi_{j_1}, k)$. Next, consider the second triple $(\pi_{i_2}, \pi_{j_2}, k)$ falls in case $(c$\,-\,$i)$. It gets mapped to $(\hat\pi_t, \hat\pi_{j_2}, k)$, different from $(\hat\pi_{i_1}, \hat\pi_{j_1}, k)$ since $\pi_t \prec \pi_s \prec \pi_{i_1}$ so $i_1 \neq t$. Finally, consider that the second triple falls in $(c$\,-\,$ii)$. It gets mapped to $(\hat\pi_{i_2}, \hat\pi_s, k)$, which is different from $(\hat\pi_{i_1}, \hat\pi_{j_1}, k)$ since $\pi_{i_2} \precsim \pi_s \prec \pi_{i_1}$ so $i_1 \neq i_2$.
If the two defects belong to $\C_1$, then we have
%
 % Let $(\pi_{i_1}, \pi_{j_1}, k)$ fall into case $(c\,$-$\,i)$. If $(\pi_{i_2}, \pi_{j_2}, k)$ also falls in $(c$\,-\,$i)$, we see that
 $i_1 = i_2 = s$, so $j_1 \neq j_2$. Thus $\psi$ maps the two defects to %and 
 $(\hat\pi_t, \hat\pi_{j_1}, k) \neq (\hat\pi_t, \hat\pi_{j_2}, k)$. 
% If instead the second triple falls in $(c$\,-\,$ii)$, we see that $(\hat\pi_t, \hat\pi_{j_1}, k) \neq (\hat\pi_{i_2}, \hat\pi_s, k)$ since $\pi_t \prec \pi_{i_2}$ so $t \neq i_2$.
 Finally, if the two defects belong to $\C_2$, 
 then we have
% consider $(\pi_{i_1}, \pi_{j_1}, k)$, $(\pi_{i_2}, \pi_{j_2}, k)$ both in case $(c$\,-\,$ii)$. Since 
$j_1 = j_2 = t$, so $i_1 \neq i_2$.  Thus $\psi$ maps the two defects to $(\hat\pi_{i_1}, \hat\pi_s, k) \neq (\hat\pi_{i_2}, \hat\pi_s, k)$. 
%Then $\psi$ is injective. 
\end{proof}

\begin{lem}\label{l: surjective}
    The map $\psi$ is surjective.  
    %In particular, we have
    %$\psi(\A) = \A'$, $\psi(\B) = \B'$, $\psi(\C_1) = \C'_1$,
    %$\psi(\C_2) = \C'_2$.
\end{lem}

\begin{proof}
%Consider a defect $(\hat\pi_i, \hat\pi_j, k) \in \D(\hat\pi)$.
For $(\hat\pi_i, \hat\pi_j, k) \in \A'$
%, then 
%($a'$) We 
we have $\hat \pi_j \prec \hat \pi_s \prec \hat \pi_t$, $i \in \{s, t\}$, so $\pi_j \prec \pi_t \prec \pi_s$. Since $i<j$, it follows that $(\pi_i, \pi_j, k)$ is defective, belongs to $\A$, and satisfies
%$, which is in case $(a)$, so 
$\psi((\pi_i, \pi_j, k)) = (\hat\pi_i, \hat\pi_j, k)$. 
Similarly, for 
$(\hat\pi_i, \hat\pi_j, k) \in \B'$
we have that
$(\pi_i, \pi_j, k)$ is defective, belongs to $\B$, and  satisfies
%$, which is in case $(a)$, so 
$\psi((\pi_i, \pi_j, k)) = (\hat\pi_i, \hat\pi_j, k)$.
%($b'$) Similar.

% ($\star$ Finish polishing the following.)

For $(\hat\pi_i, \hat\pi_j, k) \in \C'_1$ we have
%($c'\,$-$\,i$) We have 
$\hat \pi_s \precsim \hat \pi_j \prec \hat \pi_i = \hat \pi_t$ with $\hat \pi_s \neq \hat \pi_j$.
%$ and $s < t = i < j$.
Consider the triple $(\pi_s, \pi_j, k)$. 
We have that $s < t = i < j$. 
This means that $\hat \pi_j = \pi_j$ so $\pi_t \precsim \pi_j \prec \pi_s$. Since $s < j$ and $\pi_j \prec \pi_s$, we have that $(\pi_s, \pi_j, k)$ is a defect. Furthermore, these conditions imply that we have $(\pi_s, \pi_j, k) \in \C_1$.
% $j \notin \{s, t\}$ so $(\pi_s, \pi_j, k) \in \D(\pi)$ and in case $(c$\,-\,$i)$. 
By \eqref{eq:psidef}, $\psi((\pi_s, \pi_j, k)) = (\hat\pi_t, \hat\pi_j, k) = (\hat\pi_i, \hat\pi_j, k)$. 

For $(\hat\pi_i, \hat\pi_j, k) \in \C'_2$ we have
$\hat \pi_s = \hat \pi_j \prec \hat \pi_i \precsim \hat \pi_t$ with $\hat \pi_t \neq \hat \pi_i$. Consider the triple $(\pi_i, \pi_t, k)$. We have that $i < j = s < t$. This means that $\hat \pi_i = \pi_i$ so $\pi_t \prec \pi_i \precsim \pi_s$. Since $i<t$ and $\pi_t\prec \pi_i$, we have that $(\pi_i, \pi_t, k)$ is a defect. Furthermore, these conditions imply that we have $(\pi_i, \pi_t, k) \in \C_2$. By \eqref{eq:psidef}, $\psi((\pi_i, \pi_t, k)) = (\hat\pi_i, \hat\pi_s, k) = (\hat\pi_i, \hat\pi_j, k)$. 
% Since $i \notin \{s, t\}$, $(\pi_i, \pi_t, k) \in \D(\pi)$ in case $(c$\,-\,$ii)$, so $\psi((\pi_i, \pi_t, k)) = (\hat\pi_i, \hat\pi_s, k) = (\hat\pi_i, \hat\pi_j, k)$. Then $\psi$ is surjective.
\end{proof}

Given the bijectivity of $\psi$, we can now show that the map $\phi_k$
\eqref{eq:phidef}, designed to modify a path family by removing a defect at
a specified vertex $x_k$, indeed does so.
    \begin{prop}\label{p:remove one defect}
        Given star network 
    $F$ as in \eqref{eq:starnetwork}, 
    %$w \in \sn$, $d \geq 1$ 
    path family $\pi \in \Pi(F)$,
    and index $k$ such that $\pi$ has a defect at $\ctr k$,  we have that $\phi_k(\pi)=\hat\pi$ satisfies
         %the number of defects 
        \begin{equation*}\#\{(i,j) \,|\, (\hat\pi_i, \hat\pi_j, k) \text{ defective}\} = \#\{(i,j) \,|\, (\pi_i, \pi_j, k) \text{ defective}\} - 1.\end{equation*} 
        %is exactly one less than the number of defects $(\pi_i, \pi_j, k)$.
%        the number of defects $(\hat\pi_i, \hat\pi_j,p)$ with $p > k$ 
%        is exactly the number of defects $(\pi_i, \pi_j, p)$ with $p > k$.
\end{prop}
\begin{proof}
By Definition~\ref{d:defect} we consider only pairs $(i,j)$ with $i<j$.
%$i > j$ then $(\pi_i, \pi_j, i)$ and $(\hat\pi_i, \hat\pi_j, k)$ are not defective.
Let $s$, $t$
be the numbers computed in steps 1 -- 2 of 
Algorithm~$\ref{a:remove_early_defect_random_k}$,
and consider three cases of pairs $(i,j)$.

If $|\{i,j\} \cap \{s,t\}| = 0$, then  
by the discussion between \eqref{eq:st} and \eqref{eq:ijst3},
%by the cardinalities of their intersections
%\begin{equation}\label{eq:ijstcard}
%|\{i,j\} \cap \{s,t\}| \in \{0,1,2\}.
%\end{equation}
%When this cardinality is $0$, we have
%a cardinality 
%(\ref{eq:ijstcard}) 
%of $0$ implies that 
we have that $(\pi_i, \pi_j, k)$ is defective 
if and only if $(\hat\pi_i, \hat\pi_j, k)$ is defective.

If $|\{i,j\} \cap \{s,t\}| = 1$,
then by Lemmas~\ref{l: well defined} -- \ref{l: surjective},
we have a bijection between the corresponding sets $\D(\pi)$, $\D(\hat\pi)$
of defects.
%$(\pi_i, \pi_j, k)$
%correspond bijectively to

Finally, if $|\{i,j\} \cap \{s,t\}| = 2$,
%$i=s$ and $j=t$,
%while a cardinality of $2$ implies
%Similarly when the cardinality (\ref{eq:ijstcard}) is $2$,
%that 
then by Lemma~\ref{l: alt defect defn} we have that
$(\pi_i, \pi_j, k)$ 
is defective while $(\hat \pi_i, \hat \pi_j, k)$ is not.
\end{proof}

% ($\star$ Can we remove this lemma entirely? It's somewhat needed for Lemma $\ref{l: zero defects condensed}$, but seems too trivial.)
For $F$ as in \eqref{eq:starnetwork},
repeated application of the map $\phi_k$ to a path family $\pi \in \Pi_w(F)$
produces a new path family in $\Pi_w(F)$ 
having defects which match those of $\pi$ at vertices $x_{k+1}, \dotsc, x_m$,
having no defects at the vertex $x_k$,
and possibly having more defects than $\pi$ has at vertices $x_1, \dotsc, x_{k-1}$. 
%may create new defects to the left of the vertex $x_k$,
\begin{prop}\label{p: d less defects condensed}
    Fix $F$ as in \eqref{eq:starnetwork}, $v \in \sn$, $\pi \in \Pi_v(F)$, and $k \in [m]$. 
    %Suppose there exist $d$ defects of the form $(\pi_i, \pi_j, k)$.
    %Applying Algorithm~\ref{a:remove_early_defect_random_k} a total of $d$ times gives $\sigma = \phi^d_k(\pi)$ satisfying
    Then there exists $\sigma \in \Pi_v(F)$ satisfying
    \begin{enumerate}
        \item for each $p > k$, $\{(i,j) \,|\, (\sigma_i, \sigma_j, p) \text{ defective}\} = \{(i,j) \,|\, (\pi_i, \pi_j, p) \text{ defective}\}$,      
        \item $\#\{(i,j) \,|\, (\sigma_i, \sigma_j, k) \text{ defective}\} = 0$.
    \end{enumerate}
\end{prop}

\begin{proof}
    Let $d = \#\{(\pi_i, \pi_j, j) \,|\, (\pi_i, \pi_j, j) \text{ defective}\}$. Let $\sigma = \phi_k^d(\pi)$. By Proposition ~\ref{p: one less defect condensed-adjusted}, these $d$ iterations of Algorithm \ref{a:remove_early_defect_random_k} guarantee that we have $\sigma \in \Pi_v(F)$ and $(1)$. Then by Proposition \ref{p:remove one defect} these $d$ iterations give $(2)$.
\end{proof}

% \begin{proof}
%     ($\star$ Maybe mention that we can even apply $\phi_k^d$ to $\pi$, since there are $d$ defects of the form $(\pi_i, \pi_j, k)$)

%     ($\star$ Or maybe redefine $\phi_k$ so that $\phi_k(\pi) = \pi$ if there are no defects of the form $(\pi_i, \pi_j, k)$)

%     We use $\phi_k^d(\pi)$ to denote applying Algorithm~\ref{a:remove_early_defect_random_k} repeatedly $d$ times to $\pi$.
%     \begin{enumerate}
%         \item Since $\pi \in \Pi_w(F)$, then by Proposition $\ref{p: one less defect condensed-adjusted}$, we have $\phi_k(\pi), \phi_k^2(\pi), \dots, \phi_k^d(\pi) \in \Pi_w(F)$.
%         \item It is clear by Proposition $\ref{p: one less defect condensed-adjusted}$, that for each $p > k$, $\{(i,j) \,|\, (\pi_i, \pi_j, p) \text{ defective}\} = \{(i,j) \,|\, (\phi_k(\pi)_i, \phi_k(\pi)_j, p) \text{ defective}\} = \cdots = \{(i,j) \,|\, (\phi_k^d(\pi)_i, \phi_k^d(\pi)_j, p) \text{ defective}\}$.
%         ($\star$ Rewrite this?)
%         \item Since each application of $\phi_k$ decreases the number of defects at $\ctr k$ by exactly one, we clearly have that $d$ applications of $\phi_k$ drops this value to $0$. 
%     \end{enumerate}\end{proof}

We can now show, as claimed before (\ref{eq:2by2ex}), that if a star network can be covered by a path family having
$d$ defects, then it can be covered by another path family of the same
type having $d-1$ defects.
%$F = F_{[a_1,b_1]} \bullet \cdots \bullet F_{[a_m,b_m]}$,
%and path family $\pi \in \Pi_{w,d}(F)$, 
This second path family is obtained from the first by judicious applications of the maps $\phi_1,\dotsc,\phi_m$.

\begin{thm}\label{t:decrease defects by 1}
    Fix star network $F$ as in \eqref{eq:starnetwork}.  If for some $v \in \sn$ and $d \geq 1$ the set $\Pi_{v,d}(F)$ is nonempty, then the sets $\Pi_{v,0}(F),\dotsc, \Pi_{v,d-1}(F)$ are also nonempty.
\end{thm}

\begin{proof}
%\begin{proof}
    It suffices to show that for $\Pi_{v,d}(F)$ nonempty, we also have $\Pi_{v,d-1}(F)$ nonempty. Take $\pi \in \Pi_{v,d}(F)$. Suppose $k$ is the smallest value such that there exists some defect of the form $(\pi_i, \pi_j, k)$. By Proposition $\ref{p: one less defect condensed-adjusted}$, we have that $\hat\pi = \phi_k(\pi)\in \Pi_{v}(F)$ satisfies
    \begin{enumerate}[(i)]
        \item $\hat\pi \in \Pi_v(F)$,
        \item $\hat\pi$ and $\pi$ have the same numbers of defects at 
        $x_{k+1}, \dotsc, x_m$,
        \item $\hat\pi$ has one fewer defect at $x_k$ than $\pi$ has.
        \end{enumerate}
%        for each $p > k$, $\{(i,j) \,|\, (\hat\pi_i, \hat\pi_j, p) \text{ defective}\} = \{(i,j) \,|\, (\pi_i, \pi_j, p) \text{ defective}\}$,      
%        \item $\#\{(i,j) \,|\, (\hat\pi_i, \hat\pi_j, k) \text{ defective}\} = \#\{(i,j) \,|\, (\pi_i, \pi_j, k) \text{ defective}\} - 1$. 
%    \end{enumerate}
While $\hat \pi$ may have defects at $x_2,\dotsc,x_{k-1}$,
we create a new path family by removing these defects from $\hat \pi$, one vertex at a time.
%as follows.
Let $\pi^{(k)} = \hat \pi$ and construct path families 
\begin{equation*}
    \pi^{(k-1)},\, \pi^{(k-2)},\, \dots, \pi^{(2)} = \sigma
\end{equation*}
as follows. For $h = k-1, \dots, 2$, do
\begin{enumerate}
    \item Let $d_h$ be the number of defects of the form $(\pi_i^{(h+1)},\, \pi_j^{(h+1)},\, h)$ in $\pi^{(h+1)}$.
    \item Define $\pi^{(h)} = \phi_h^{d_h}(\pi^{(h+1)})$.
\end{enumerate}
By Proposition \ref{p: d less defects condensed}, each path family $\pi^{(h)}, h = k-1, \dots, 2$ has type $v$, has no defects at $\ctr h$, and has the same number of defects at $\ctr{h+1}, \dots, \ctr m$ as $\pi^{(h+1)}$. It follows that $\sigma = \pi^{(2)}$, like $\pi$, has no defects at $\ctr 2, \dots, \ctr{k - 1}$. By Observation \ref{o:defect}, neither $\sigma$ nor $\pi$ has a defect at $\ctr 1$. Furthermore, $\sigma$ has one fewer defect at $\ctr k$ than $\pi$ has, and has exactly the same number of defects at $\ctr{k+1}, \dots, \ctr{m}$ as $\pi$ has. Thus we have $\sigma \in \Pi_{v, d-1}(F)$, as desired.
% Let $\pi^{(k)} = \hat\pi$ and
% let $d_{k-1}$ be the number of defects of the form $(\pi^{(k)}_i, \pi^{(k)}_j, k-1)$ in $\pi^{(k)}$.
% Remove all of these defects by applying $\phi_{k-1}^{d_{k-1}}$ to $\pi^{(k)}$,
% and call the resulting path family $\pi^{(k-1)}$. By Proposition \ref{p: d less defects condensed}, $\pi^{(k-1)}$ has type $w$, has no defects at $\ctr{k-1}$, and has the same number of defects at $\ctr k, \dots, \ctr m$ as $\hat\pi$ has. 
% Continue this way, removing defects at vertices $x_{k-2}, \dotsc, x_2$ until
% producing a path family $\sigma$ having one fewer defect at $x_k$ than $\pi$ has,
% and the same number of defects at $x_{k+1},\dotsc, x_m$ as $\pi$ has.
\end{proof}

If a star network $F$ can be covered by a path family of type $w$ having 0 defects, then this path family is unique.
% Such a path family is in fact unique.

\begin{cor}\label{c: unique zero defect-condensed}
    For any fixed $v \in \sn$ and star network $F$ as in \eqref{eq:starnetwork}, 
    % we have $|\Pi_{w,0}(F)| \leq 1$. 
    the number of path families of type $v$ and having $0$ defects is at most one. If $v = e$, then this number is exactly one.
\end{cor}
% \begin{proof}
% % ($\star$ Update this proof so that it applies to condensed concatenations.) 
% Let $F=F_{[a_1,b_1]}\bullet\cdots \bullet F_{[a_m,b_m]}$. If $\Pi_w(F)$ is nonempty, then by Theorem~\ref{t:decrease defects by 1}
% %Lemma~\ref{l: zero defect} 
% we can choose $\pi \in \Pi_{w,0}(F)$. Since $\pi$ is a path family of type $w$, each path $\pi_i$ terminates at $\snk i m (F)$. All paths which don't enter the internal vertex of $F_j$, for any $1 \leq j \leq m$, clearly enter and terminate at the same source and sink vertex of $F_j$. Those which do enter the internal vertex of $F_j$ must enter through the source vertices of $F_j$ in sorted order (including those share an edge?), otherwise, some pair which enters the vertex out of order would create a defect, which $\pi$ has none. Since every path $\pi_i$ terminates at $\snk i m$, there's only one source vertex it could have entered from to ensure $0$ defects, as described by the sorted property above (Lemma~\ref{l:condensed zero defects}?). Then $\snk i {m-1}$ is set, and so must be $\src i {m-1}$. Applying this property backwards all the way to $F_1$, we see that $\pi$ is uniquely determined by its type $w$ and its planar network $F$, thus $|\Pi_{w,0}(F)| \leq 1$.

% It is clear that $\Pi(F)$ contains exactly one path family with no crossings. This path family has type $e$ and no defects.
% \end{proof}

% ($\star$ Look over text from previous meeting and decide structure of lemmas/claims)
\begin{proof}
If $\Pi_v(F)$ is empty, then clearly no path family of type $v$ in $F$ has $0$ defects.  Assume therefore that $\Pi_v(F)$ is nonempty.  By Theorem~\ref{t:decrease defects by 1} this set contains at least one path family having $0$ defects.

Consider the case $v=e$.  If $\pi \in \Pi_e(F)$ has two paths which cross, then those paths must later cross again, creating a defect.  Thus there is a unique path family in $\Pi_e(F)$ having no defects: the path family having no crossings.  

Suppose now that the claim is true for 
%path families in $\Pi_w(F)$, where 
$v$ of length $0,\dotsc, \ell-1$ and consider $v$ of length $\ell$.
%with a crossing must have acontains exactly one path family with no crossings and therefore no defects.  Any two paths 
%Suppose we have $\pi$, $\sigma$ in $\Pi_e(F)$ both with zero defects.
Let $s_i$ be a left descent of $v$, i.e.,
$s_iv < v$, equivalently, $v_i > v_{i+1}$. 
Choose
$\pi \in \Pi_{v, 0}(F)$ 
and let $(\ctr{p_1}, \dots, \ctr{p_2})$ be the unique crossing of 
$\pi_i, \pi_{i+1}$. 
Define $\pi' \in \Pi_{s_iv}(F)$ 
by swapping $\ctr{p_2}$-to-sink subpaths of $\pi_i, \pi_{i+1}$.  
We will show that the map $\pi \mapsto \pi'$ is a bijection from $\Pi_{v, 0}(F)$ to $\Pi_{s_iv, 0}(F)$. By induction, the uniqueness of $\pi \in \Pi_{v, 0}(F)$ will follow.
% We claim that $\pi'$ has $0$ defects. 

    First we claim that $\dfct(\pi') = 0$. To see this, suppose  that $(\pi'_h, \pi'_j, q)$ is a defect of $\pi'$ and consider the sets $\{h, j\}, \{i, i+1\}$.
    The sets cannot be equal since $\pi'_i, \pi'_{i+1}$ do not cross. 
    On the other hand, the sets cannot be disjoint, since the equalities $\pi_h = \pi'_h$ and $\pi_j = \pi'_j$ would imply a defect in $\pi$. We therefore have
    \begin{equation}\label{eq: hiji+1}
        h < i \text{ and } j \in \{i, i+1\} \quad \text{or} \quad h \in \{i, i+1\} \text{ and } i+1 < j.
    \end{equation}
    % Consider the case where $h < i$ and $j \in \{i, i+1\}$.
    
    If $h < i = j$, then $(\pi_h', \pi_i', q)$ is a defect of $\pi'$. By Lemma~\ref{l: alt defect defn} we have $\pi'_i \prec \pi'_h$. If $q \leq p_2$ then this defect in $\pi'$ implies the triple $(\pi_h, \pi_i, q)$ to be defective in $\pi$;
    if $q > p_2$, then 
    %the triple
    %We must have $q > p_2$,
%    can't have $q \leq p_2$ 
%because otherwise this would imply the existence of a defect in $\pi$.  
it follows that $\pi_{i+1} \prec \pi_h$ and $(\pi_h, \pi_{i+1}, q)$ is defective in $\pi$.
If on the other hand $h < i$ and $i+1 = j$, then $(\pi_h', \pi_{i+1}', q)$ is a defect of $\pi'$ and $\pi'_{i+1} \prec \pi'_h$. Again, we have $q > p_2$ and it follows that $\pi_{i} \prec \pi_h$. By Lemma \ref{l: alt defect defn}, $(\pi_h, \pi_{i}, q)$ is a defect of $\pi$. Since $\dfct(\pi) = 0$, all of these cases lead to a contradiction. For the remaining cases in \eqref{eq: hiji+1}, the argument is similar.

   % ($\star$ Do we need to say anything else about other triples $(\pi'_j, \pi'_k, q)$ for $j, k \notin \{i, i+1\}$ or $\{j, k\} = \{i, i+1\}$?)
% \end{proof}

%\begin{lem}
    Now we claim that the map $\pi \mapsto \pi'$ from $\Pi_{v, 0}(F)$ to $\Pi_{s_iv, 0}(F)$ is invertible:
%\end{lem}
%\begin{proof}
    its inverse is given by finding the rightmost component $(\ctr{p_1}, \dots, \ctr{p_2})$ of intersection of $\pi_i', \pi_{i+1}'$ and making it into a crossing. The lack of defects in $\pi$ implies that $(\ctr{p_1}, \dots, \ctr{p_2})$ is the rightmost component of intersection of $\pi_i, \pi_{i+1}$ as well. By the inductive hypothesis, $\pi'$ is the unique path family in $\Pi_{s_iv, 0}(F)$, so $\pi$ is the unique path family in $\Pi_{v, 0}(F)$.
\end{proof}

\begin{cor}
\label{c:onedefect-condensed}
    Fix star network $F$ as in \eqref{eq:starnetwork}.
    %Let $F = F_{[a_1,b_1]} \bullet \cdots \bullet F_{[a_m,b_m]}$.
    If $\Pi_e(F)$ contains more than one path family, then it contains a path family having exactly one defect.
\end{cor}
\begin{proof}
    Suppose $\Pi_e(F)$ contains at least two path families.
    By Corollary~\ref{c: unique zero defect-condensed}, one element of $\Pi_e(F)$ %of these 
    is the unique 
    %(noncrossing) 
    element of $\Pi_{e,0}(F)$.
    Choose another path family in $\Pi_e(F)$ and let $d \geq 1$ be the 
    number of its defects.  By Theorem~\ref{t:decrease defects by 1},
    the set $\Pi_{e,1}(F)$ is nonempty.
\end{proof}

\section{Main Result}\label{s:nofactor}

Every polynomial in $\mathbb N[q]$ with constant term $1$ arises as a Kazhdan--Lusztig polynomial~\cite{PoloConstrArb}.  Gaetz--Gao~\cite{GGMinPower} studied the sequences of coefficients in these polynomials, especially coefficients equal to %and the appearances of 
$0$ 
%as a coefficient 
between other
%the first and last 
nonzero coefficients.
%($\star$ Mention the connection between Kahzdan--Lusztig polynomials and smoothness in Schubert varieties.)

Define a function $\singdeg: \sn \rightarrow \mathbb N \cup \{\infty\}$ by
\begin{equation}\label{eq:gammadef}
\singdeg(w) = \begin{cases}
    \min\{ k > 0 \,|\, 
    %\text{$q^k$ has nonzero coefficient in $P_{v,w}(q)$}\}
    \text{coefficient of $q^k$ in $P_{e,w}(q)$ is nonzero}\} 
    &\text{if $P_{e,w}(q) \neq 1$},\\
    \infty &\text{if $P_{e,w}(q) = 1$}.
    \end{cases}    
\end{equation}
This is %essentially 
a lower bound on degrees for which
%the first degree in which 
Poincar\'e duality fails in the Schubert variety $X_w$~\cite{KLSchub}, and can be computed in terms of patterns in $w$ and a related definition.
%($\star$ say more about smoothness and singularities and Poincar\'e duality.)
%($\star$ Who first considered this?)
Specifically, given $w \in \sn$
%and $w$ 
not avoiding the pattern $3412$,
%and $u \in \mfs{k}$, 
define the %Cortez 
%$(b,u)$-
3412-{\em gap} of $w$
by
%to be the number
\begin{equation}\label{eq:cortezgap}
\pgap(w) = \min\{ w_{i_1} - w_{i_4} \,|\, \text{subword $w_{i_1} w_{i_2} w_{i_3} w_{i_4}$ 
%a subword of $w$ 
matches the pattern $3412$} \}.
\end{equation}
%($\star$ Improve this definition.)
%($\star$ Was this Cortez's idea?  See reference in Gaetz--Gao bibliography.  At the very least, mention that she and Wu used it.)
For $w$ \avoidingp, we have $\singdeg(w) = \infty$.  Otherwise, we can compute $\singdeg(w)$ in terms of $\pgap(w)$ as follows
%The Gaetz--Gao result
~\cite[Thm.\,1.6]{GGMinPower}.
%relates pattern avoidance in $w$ to the smallest positive power of $q$ that appears in the Kazhdan--Lusztig polynomial $P_{e,w}(q)$, assuming this polynomial is nonconstant.
\begin{thm}\label{t:GaetzGao}
    For $w$ not \avoidingp{}
    %the pattern $3412$ 
    %and satisfying $P_{e,w}(q) \neq 1$, 
    we have
    \begin{equation*}
        \singdeg(w) = \begin{cases}
            \pgap(w) &\text{if $w$ avoids the pattern $4231$},\\
            1 &\text{otherwise}.
            % &\text{if $w$ 
            % does not avoid
            % %contains a subword which matches 
            % the pattern $4231$},\\
            % \pgap(w) &\text{otherwise}.
        \end{cases}
    \end{equation*}
\end{thm}

%($\star$ Do we have the following corollary? Or maybe that corollary isn't so much simpler than the theorem.)

%\begin{cor}\label{c:gamma>1}
%    We have $\singdeg(w) > 1$ if $w$ avoids the patterns $4231$, $34125$, $34512$, $53412$, $45123$, $45231$, $14523$, and does not avoid 
%    %contains a subword which matches 
%    the pattern $45312$.
%\end{cor}
For example, consider the permutation $45312 \in \mfs 5$, %and the Kazhdan--Lusztig polynomial $P_{e,45312}(q) = 1+q^2$.
%Since $45312$
which 
avoids the pattern $4231$ and has
only the subword $4512$ matching the pattern $3412$. Since $\pgap(45312)=2$, Theorem~\ref{t:GaetzGao} implies that the coefficient of $q$ in $P_{e,45312}(q)$ is $0$ and that the coefficient of $q^2$ is not.
%and thus we
%have $\pgap(45312) = 4-2 = 2$.
%It follows that $P_{e,45312}(q)$ has the form $1 + 0q + a_2q^2 + \cdots + a_k q^k$ with $a_2 > 0$
%\in \mathbb N[q]$ 
%for some $k \geq 2$.
This is consistent with the fact that $P_{e,45312}(q) = 1+q^2$. (See \cite[p.\,75]{BilleyLak}.)

%Our theorem says that f
For each permutation $w$ having $\singdeg(w) > 1$,
%$w$ satisfying the conditions of Corollary~\ref{c:gamma>1},
the Kazhdan--Lusztig basis element $\smash{\wtc wq}$ has no reversal factorization 
%as in 
\eqref{eq:reversalfactor} and therefore is not graphically representable 
by a star network, in the sense of (\ref{eq:Gtozhnq}).
\begin{thm}\label{t:main}
    For $w \in \sn$ 
    %having $\singdeg(w) > 1$,
    %avoiding the patterns $4231$, $34125$, $34512$, $53412$,
    %and 
    avoiding the pattern $4231$, 
    not avoiding the pattern $3412$,
    and having $\pgap(w) > 1$, 
    the Kazhdan--Lusztig basis element $\wtc wq$ 
    %is not graphically representable by a star network, and therefore 
    has no reversal factorization.
\end{thm}
\begin{proof}
Fix $w$ as above with $k = \pgap(w)$ and suppose that the star network $F$
graphically represents $\smash{\wtc wq}$ as an element of $\hnq$,
\begin{equation}\label{eq:main}
\wtc wq = \sum_{v \in \sn} \sum_{d \geq 0} |\Pi_{v,d}(F)|
%\sum_{\pi \in \Pi_{v,d}(F)} 
q^d T_v = \sum_{v \leq w} P_{v,w}(q) T_v.
\end{equation}
%Then the coefficient of $T_e$ in (\ref{eq:main}) is $P_{e,w}(q)$.
Since the constant term of $P_{e,w}(q)$ is $1$,
%(or by Corollary~\ref{c: unique zero defect}), 
we have $|\Pi_{e,0}(F)| = 1$.
By our definition of $k$
%assumptions on $w$ 
and Theorem~\ref{t:GaetzGao}, the coefficients of $q, \dotsc, q^{k-1}$
in $P_{e,w}(q)$ are $0$, while the coefficient of $q^k$ is positive.
In particular, we have the cardinalities $|\Pi_{e,1}(F)| = \cdots = |\Pi_{e,k-1}(F)| = 0$ and $|\Pi_{e,k}(F)| > 0$, which
% and $|\Pi_e(F)| \geq 2$.
% %which
% But this combination
contradict
Theorem~\ref{t:decrease defects by 1}.
%Corollary~\ref{c:onedefect}.
%    ($\star$ proof by contradiction)
\end{proof}
By Theorem~\ref{t:decrease defects by 1}, reversal factorization of $\wtc wq$ implies that {\em none} of the Kazhdan--Lusztig polynomials $P_{v,w}(q)$ has internal coefficients equal to zero.
\begin{thm}
    If $\wtc wq$ has a reversal factorization,
    %s as in \eqref{eq:reversalfactor},
    then for every $v < w$ there exists $k = k(v)$
    in $\mathbb N$ 
    such that we have $P_{v,w}(q) = 1 + a_1q + \cdots + a_kq^k$ with $a_1,\dotsc, a_k > 0$.
\end{thm}
%\begin{proof}
%    (Condensed concatenation version of previous proofs that a path family with $d$ defects implies the existence of another with $d-1$ defects.)
%\end{proof}
It is easy to show that the inequality $\pgap(w) > 1$ implies that $w$ does not avoid the pattern $45312$. It is also easy to show that no star network graphically represents
%, the impossibility of representing
$\wtc{453129786}q$, even though the subword $9786$ matches the pattern $4231$. Indeed, some limited experimentation~\cite{DatSkan} suggests that avoidance of 
the pattern $45312$ is important and avoidance of the pattern $4231$ is unimportant in the classification of permutations $w$ for which
$\smash{\wtc wq}$ is graphically representable by a star network.
We conjecture the following partial answer to \cite[Quest.\,4.5]{SkanNNDCB}.
%presence 
%Moreover, some experimentation in \cite{DatSkan} suggests that 
%Mention that $P_{e,45312} = 1+q^2$.  (See \cite[p.\,75]{BilleyLak} for %partial table of Kazhdan--Lusztig polynomials when $n = 5$.)
\begin{conj}
    If $w \in \sn$ does not avoid the pattern $45312$, then the Kazhdan--Lusztig basis element $\wtc wq$ has no reversal factorization.
    %\eqref{eq:reversalfactor}.
\end{conj}

For small $n$, it is possible to explicitly enumerate the permutations $w \in \sn$ avoiding the pattern $45312$ and to verify that each corresponding Kazhdan--Lusztig basis element $\wtc wq$ has a reversal factorization~\cite{DatSkan}.
Furthermore, given 
%$w = w_1 \cdots w_n \in \sn$ and 
planar network $F$ which graphically represents
$\wtc wq$, the reflection of $F$ in a horizontal line,
the reflection of $F$ in a vertical line,
and the rotation of $F$ by $180^\circ$ %radians
can graphically represent up to three additional elements 
$\wtc{w_0ww_0}q$, $\wtc{w^{-1}}q$, $\wtc{w_0w^{-1}w_0}q$, 
%$\wtc vq$: $v = w_0ww_0$, $v=w^{-1}$,
%and $v = w_0w^{-1}w_0$, 
respectively, where $w_0$ is the reversal $s_{[1,n]}$.
For example, when $n= 5$, all Kazhdan--Lusztig basis elements except for $\wtc{45312}q$ have reversal factorizations (which need not be unique).
If $w$ avoids the pattern $321$, a factorization of $\wtc wq$ is given by Theorem~\ref{t:BWHex}.
If $w$ \avoidsp, a factorization of $\wtc wq$ is given by
Theorem~\ref{t:SkanSmooth}. For the remaining twenty-three permutations $w$, we have the following 
factorizations and
graphical representations of $\wtc wq$.

\vspace*{8mm}

\hhhsp
\begin{center}
\newcolumntype{R}{>{$}c<{$}}
%\begin{tabularx}{166.5mm}{|R|R|R|R|R|}%
\begin{tabularx}{151.5mm}{|R|R|R|R|}%
\hline
%& & & &\\
& & &\\
& & &\\
w
& \phmat \begin{matrix}
    \phsum \mbox{a reversal} \phsum \\
    \mbox{factorization of } \wtc wq\end{matrix}\phmat
%& \mbox{a factorization of } \wtc wq
& \nTksp \begin{matrix} 
\phsum \mbox{a graphical} \phsum \\  
\quad\mbox{representation
%& \begin{matrix} \nTksp \phsum \nTksp \mbox{a graphical} \phsum \nTksp \\  \mbox{representation}\\ 
of } \wtc wq \end{matrix}
& \begin{matrix} \phint \ntksp \mbox{related permutations} \ntksp \phint \\ w_0ww_0,\ w^{-1}\ntksp,\ w_0w^{-1}w_0 \nTksp \nTksp \phc \nTksp\end{matrix}
%& \theta(C'_w(1)) \in \mathbb N \mathrm{?}  
%& \begin{matrix} \mbox{interpretation of} \\ 
%\mbox{$\theta(C'_w(1))$ as $|R|^{\phantom{\hat T}}\nTksp$ for} \\ 
%\mbox{$w$ avoiding $312^{\phantom{\hat T}}\nTksp$?}\end{matrix} \\
%\mbox{$\theta(C'_w(1))$ as $|R|^{\phantom{\hat T}}\nTksp$} \\ 
%\mbox{for $w^{\phantom{\hat T}}\nTksp$ avoiding} \\
%\mbox{$3412^{\phantom{\hat T}}\nTksp$ and $4231$?}\end{matrix} 
\\
& & &\\
& & &\\
\hline 
& & &\\
& & &\\
42315 & 
\wtc{s_{[1,2]}}q \wtc{s_{[2,4]}}q \wtc{s_{[1,2]}}q
& \Fdbcae
& 15342 \\
& & & \\
& & & \\
 \hline
 & & &\\
 & & &\\
 52314 & 
 \wtc{s_{[1,2]}}q \wtc{s_{[2,4]}}q \wtc{s_{[1,2]}}q \wtc{s_{[4,5]}}q 
 & \Febcad
 & 25341, 42351, 51342 \\
 & & & \\
 & & & \\
\hline
& & &\\
& & &\\
35142 & 
\phm 
\wtc{s_{[2,3]}}q \wtc{s_{[1,2]}}q \wtc{s_{[3,5]}}q \wtc{s_{[2,3]}}q 
\phm
& \Fceadb
& 42513 \\
& & & \\
& & & \\
\hline
& & &\\
& & &\\
35241 & 
\phm 
\wtc{s_{[2,3]}}q \wtc{s_{[3,5]}}q \wtc{s_{[1,3]}}q
%\frac 1{1+q}
%\wtc{s_{[4,5]}}q \wtc{s_{[2,4]}}q \wtc{s_{[1,3]}}q \wtc{s_{[4,5]}}q \phm
& \Fcebda
& 52413, 53142, 42531 \\
& & & \\
& & & \\
\hline
\end{tabularx}
\end{center}
%\vspace{2mm}
\ssp

\newpage
\hhhsp
\begin{center}
\newcolumntype{R}{>{$}c<{$}}
%\begin{tabularx}{166.5mm}{|R|R|R|R|R|}%
\begin{tabularx}{153.5mm}{|R|R|R|R|}%
\hline
%& & & &\\
& & &\\
& & &\\
w
%: \csn \rightarrow \mathbb C 
&  \phmat \begin{matrix}
    \phint \mbox{a reversal} \phint \\
    \phint \mbox{factorization of } \wtc wq\phint \end{matrix}
%& \mbox{a factorization of } \wtc wq
& \nTksp \nTksp \begin{matrix} 
\phsum \mbox{a graphical} \phsum \\  
\phsum \mbox{representation of } \wtc wq \phsum
\end{matrix} \nTksp \nTksp
%\mbox{planar network}
& \begin{matrix} \mbox{related permutations} \\ w_0ww_0,\ w^{-1}\ntksp,\ w_0w^{-1}w_0 \nTksp \nTksp \phc \nTksp \end{matrix}
\\
\hline
& & &\\
& & &\\
35412 & 
\frac 1{1+q}
\wtc{s_{[2,4]}}q \wtc{s_{[3,5]}}q \wtc{s_{[1,2]}}q \wtc{s_{[2,3]}}q 
& \Fcedab
& 45213, 45132, 43512 \\
& & & \\
& & & \\
\hline
& & &\\
& & &\\
52431 & 
\wtc{s_{[1,2]}}q \wtc{s_{[2,5]}}q
\wtc{s_{[1,2]}}q
& \Febdca
& 53241 \\
& & & \\
& & & \\
\hline
& & &\\
& & &\\
53421 & 
\frac 1{1+q}
\wtc{s_{[3,5]}}q \wtc{s_{[1,4]}}q \wtc{s_{[4,5]}}q
& \Fecdba
& 54231 \\
& & & \\
& & & \\
\hline
& & &\\
& & &\\
45231 & 
\wtc{s_{[2,3]}}q \wtc{s_{[3,5]}}q \wtc{s_{[1,3]}}q \wtc{s_{[3,4]}}q 
& \Fdebca
& 53412 \\
& & & \\
& & & \\
\hline
& & &\\
& & &\\
52341 & 
%\phj
\wtc{s_{[1,2]}}q \wtc{s_{[4,5]}}q \wtc{s_{[2,4]}}q \wtc{s_{[1,2]}}q \wtc{s_{[4,5]}}q
%\phj
& \Febcda
& $---$ \\
& & & \\
& & & \\
\hline
\end{tabularx}
\end{center}
%\vspace{2mm}
\ssp

\vspace*{8mm}

Of course, it would be interesting to find larger sets of permutations $w$ in $\sn$ for which $\wtc wq$ has a reversal factorization,
and for which $\wtc wq$ has no reversal factorization.

%$\sum_A$ $\sum_\sum$ $\int$ $\int_A$ $\int_\int$ $\begin{matrix} 1\\1\end{matrix}$
\section{Acknowledgements}
The authors are grateful to Sara Billey, Ashton Datko, and Greg Warrington for help with computations, and to Grant Barkley for helpful conversations.
\bibliography{my}
\end{document}